\documentclass[11pt]{amsart}
\usepackage{mathrsfs}
\usepackage[letterpaper]{geometry}
\usepackage{graphicx}
\usepackage{amssymb}
\usepackage{epstopdf}
\usepackage{amsmath,amscd}
\usepackage{amsthm}
\usepackage{url,verbatim}
\usepackage{enumitem}
\usepackage{subfig}

\usepackage{setspace}

\calclayout
\RequirePackage[colorlinks,citecolor=blue,linkcolor=blue,urlcolor=blue]{hyperref}
\usepackage{breakurl}
\usepackage{float}
\theoremstyle{plain}
\newtheorem{theorem}{Theorem}

\newtheorem{lemma}[theorem]{Lemma}
\newtheorem{proposition}[theorem]{Proposition}
\newtheorem{corollary}[theorem]{Corollary}
\newtheorem{example}[theorem]{Example}

\renewcommand\ell{l}

\newcounter{mycount}

\numberwithin{equation}{section}
\numberwithin{theorem}{section}
\numberwithin{figure}{section}

\newcommand{\wbar}\widebar

\DeclareMathAlphabet{\mathpzc}{OT1}{pzc}{m}{it}

 \usepackage{soul}
\usepackage[dvipsnames]{xcolor}

\title{Matchings on Random Regular Hypergraphs}
\author{Zhongyang Li}
\address{(ZL) Department of Mathematics,
University of Connecticut, Storrs, Connecticut 06269-3009, USA} \email{zhongyang.li@uconn.edu}
\urladdr{\url{http://www.math.uconn.edu/~zhongyang/}}

\begin{document}

\begin{abstract}
We study the monomer--dimer partition function on the configuration model of
random $d$-regular, $l$-uniform hypergraphs.  For fixed $d,l\ge2$, we prove
quenched free-energy limits in explicit parameter regimes.  The proof combines
fixed-density first-moment asymptotics, a two-overlap second-moment variational
analysis, and a subgraph-conditioning argument for the short cycles of the
incidence structure.  The main technical point is to identify regimes in which
the replica-symmetric saddle is the unique global maximizer of the second-moment
rate function.  In those regimes the normalized logarithm of the total matching
partition function converges in probability to an explicit variational value.  We
also prove the corresponding result for the weighted partition function whenever
the maximizing density lies in the verified replica-symmetric region, give an
additional checkable criterion for that region, and record a first-moment upper
tail estimate for the maximum matching size.
\end{abstract}

\maketitle

\section{Introduction}

Random sparse factor graphs provide a natural setting for Gibbs measures with
hard constraints.  A central quantity in such models is the quenched free energy,
the normalized logarithm of the partition function.  For locally tree-like
models, non-rigorous cavity predictions and rigorous tree recursions often point
to a replica-symmetric formula, but a finite random graph may still have
non-negligible fluctuations caused by short cycles.  This paper studies this
problem for matchings, or monomer--dimer configurations, on random regular
uniform hypergraphs.  We prove convergence in probability of the quenched free
energy in explicit finite-degree regimes by combining a fixed-density
second-moment analysis with subgraph conditioning.

Counting matchings has been a central problem in combinatorics and statistical
mechanics since at least the 1960s; see \cite{ER66,HL70,HL72}.  Perfect
matchings, or dimer configurations, form an important special case.  Although the
number of perfect matchings in a planar, or nearly planar, graph can often be
computed by determinantal formulae \cite{Kas61,TF61,RK09}, counting all
matchings in two-dimensional graphs is more subtle; see \cite{JM87} and the
algorithms in \cite{CRS96,GK09}.  The full matching partition function is in
some respects more stable than the perfect-matching partition function: for
example, the normalized logarithm of the number of matchings is continuous along
weakly convergent graph sequences, whereas the corresponding statement for
perfect matchings fails even for bipartite regular graphs \cite{ACFK16}.

We work in the configuration model.  Let $d,l\ge2$.  A $(d,l)$-regular
hypergraph $G=(V,E,H)$ consists of a set $H$ of half-edges, a set $E$ of
hyperedges, each containing $l$ distinct half-edges, and a set $V$ of vertices,
each containing $d$ distinct half-edges.  Each half-edge belongs to exactly one
hyperedge and to exactly one vertex.  We allow two distinct half-edges of the
same hyperedge to belong to the same vertex, and we also allow two distinct
hyperedges to share more than one vertex.  Thus the model is not restricted to
simple hypergraphs.

A subset $M\subset E$ is a matching if no vertex is incident to more than one
half-edge belonging to the hyperedges in $M$.  Equivalently, all $l|M|$
half-edges contained in the hyperedges of $M$ are incident to distinct vertices.
This convention rules out, in particular, a selected hyperedge whose own
half-edges meet the same vertex twice.

Let $m$ be the number of hyperedges and assume $lm/d\in\mathbb N$.  We take
$E=[m]$ and $H=[lm]$, with half-edge $i$ belonging to the hyperedge
$\lceil i/l\rceil$.  A configuration is obtained by partitioning $H$ into $lm/d$
vertex-classes of size $d$.  For counting purposes it is convenient to use the
equivalent model in which each vertex-class is given a cyclic order.  Since each
underlying partition has exactly $((d-1)!)^{lm/d}$ such representations, this
convention does not change the induced uniform distribution on underlying vertex
partitions.  Under this convention the number of configurations is
\begin{equation}\label{dlr}
  \prod_{j=0}^{lm/d-1}\prod_{k=1}^{d-1}(k+jd)
  =\frac{(lm)!}{d^{lm/d}(lm/d)!} .
\end{equation}
For fixed $d$ and $l$, Stirling's formula gives
\[
  \frac{(lm)!}{d^{lm/d}(lm/d)!}
  =e^{O(1/m)}\sqrt d\left(\frac{lm}{e}\right)^{lm(d-1)/d} .
\]
Let $\Omega_{m,d,l}$ denote the set of these configurations, and let
$\mathcal G_{m,d,l}$ be uniformly distributed on $\Omega_{m,d,l}$.  All limits in
this paper are taken as $m\to\infty$ through values for which $lm/d\in\mathbb N$.

The case $l=2$ is the monomer--dimer model on a random regular graph.  Its free
energy was computed in \cite{AC14}, and the density of maximum matchings was
studied via local weak convergence in \cite{EL10,BLS13}; see also
\cite{AS04,DM10}.  The hypergraph case $l\ge3$ is different at the level of the
second moment.  At a fixed density, an ordered pair of matchings is described by
two overlap parameters, and the second moment reduces to a two-dimensional
variational problem.  The replica-symmetric saddle is the point predicted by two
independent random matchings, but it is not automatic that this point is the
global maximizer.  A main contribution of the paper is to prove this global
maximality in explicit regimes, and then to use short-cycle conditioning to turn
the second-moment information into a quenched free-energy limit.

The proof is model-specific.  After the first-moment calculation, we compute an
exact fixed-density second-moment sum and analyze its rate function.  In the
verified replica-symmetric regimes, the unique global maximum is non-degenerate,
so a two-dimensional Laplace estimate gives the second-moment asymptotics.  The
remaining fluctuations come from cycles in the incidence structure.  We compute
the limiting joint law of finitely many short-cycle counts, both unconditionally
and conditioned on a fixed matching, and apply subgraph conditioning.  This is
analogous in spirit to the use of subgraph conditioning for Hamiltonian cycles in
random regular graphs \cite{RW94}, and to related applications of moment and
subgraph-conditioning methods in \cite{AM06,AN05,GSV16}, but the conditional
cycle calculation here reflects the hypergraph matching constraint.

We also distinguish the present result from more general approaches to sparse
factor models.  Dembo, Montanari and Sun \cite{DMS13Factor} prove Bethe-type
free-energy formulae for locally tree-like graph sequences under uniqueness
hypotheses, while Coja-Oghlan and Perkins \cite{COP18} prove a general belief
propagation formula for replica-symmetric random factor graph models, including
random regular factor graphs.  For hypergraph matchings, Song, Yin and Zhao
\cite{SYZ19} study approximate counting and correlation decay up to the
uniqueness threshold.  Our contribution is complementary to these frameworks:
we give a finite-degree, matching-specific verification of the
replica-symmetric saddle in explicit regions of the random regular configuration
model, together with the short-cycle correction needed for quenched convergence
in probability.

There is also a tree-recursion perspective.  A Gibbs measure for matchings on the
infinite $d$-regular, $l$-uniform hypertree may or may not be unique, and it may
or may not satisfy strong spatial mixing (SSM).  In regimes where uniqueness or
correlation decay is available, one can often obtain the limiting free energy
from the tree recursion; see, for example, \cite{DW06,LL13,BST16,CZ12,BGS19}.
The criteria below are complementary finite-degree moment conditions.  They are
not intended to replace the uniqueness or correlation-decay theory, and we make
no claim that they describe a sharp threshold.  Rather, they verify directly, for
the random regular configuration model, that the replica-symmetric saddle controls
the relevant moment calculation.  The appendix shows, however, that the additional
criterion in Theorem~\ref{th12} can hold for only finitely many pairs $(d,l)$.

When $d=2$, the present model is equivalent to the independent-set model on a
random $l$-regular graph at activity $1$.  Maximum-size fluctuations for that
model have been studied when $l$ is large \cite{SSD}.  Free energies of more
general vertex models on graphs have been studied, for example, in
\cite{ZL11,GrL18,GrL16,GrL17}.  The present paper keeps the hypergraph matching
structure explicit, because the overlap geometry and the short-cycle correction
are both specific to this model.

We now state the main results.  All free energies in the paper are normalized by
$m$, the number of hyperedges; normalization by the number of vertices differs by
the deterministic factor $d/l$.  For an integer $h\ge0$, let $Z_h$ be the number
of matchings with exactly $h$ hyperedges, and let
$Z=\sum_{h=0}^{\lfloor m/d\rfloor}Z_h$ be the total number of matchings.

\begin{theorem}\label{thm}
\begin{enumerate}
\item Let $\beta_*$ be the unique solution in $(0,1/d)$ of
\begin{equation}\label{eq1}
  (1-d\beta)^l=\beta(1-\beta)^{l-1}.
\end{equation}
Define
\begin{equation}\label{pdb}
  \Phi_{d,l}(\beta)
  =-\beta\log \beta+(l-1)(1-\beta)\log(1-\beta)
   -\frac ld(1-d\beta)\log(1-d\beta).
\end{equation}
Then
\[
  \lim_{m\to\infty}e^{-m\Phi_{d,l}(\beta_*)}\mathbb E Z
  =\sqrt{\frac{1-\beta_*}{1+(ld-d-l)\beta_*}} .
\]

\item For $l\ge3$, define
\begin{align}
  L_1&:=\frac{1}{dl-d-l+2}
       =\frac{1}{(d-1)(l-1)+1},\label{eq:intro-L1}\\
  L_2&:=\min\biggl\{
    \frac1d\left(1-\sqrt{\frac{d-1}{d^{l/(l-1)}-1}}\right),\notag\\
  &\hspace{25mm}
    \frac{dl+l^2-2l-d+1}{2dl^2-dl},
    \frac{1}{1+\sqrt{(d-1)(l-1)}}
  \biggr\},\label{eq:intro-L2}\\
  L_{\rm cert}&:=\max\{L_1,L_2\},\label{eq:intro-Lcert}\\
  x_{\rm cert}&:=\exp\{-\Phi'_{d,l}(L_{\rm cert})\}.\label{eq:intro-xcert}
\end{align}
Assume that either
\begin{enumerate}
\item $l=2$, or
\item $l\ge3$ and $\beta_*\le L_{\rm cert}$.
\end{enumerate}
Equivalently, in the case $l\ge3$, one may write the condition as
$x_{\rm cert}\ge1$.  Then
\[
  \frac1m\log Z\longrightarrow \Phi_{d,l}(\beta_*)
\]
in probability.

\item Let $x>0$ and define the weighted partition function
\begin{equation}\label{dzx}
  Z(x):=\sum_{h=0}^{\lfloor m/d\rfloor} Z_h x^h .
\end{equation}
Let $\beta_*(x)$ be the unique solution in $(0,1/d)$ of
\[
  \Phi'_{d,l}(\beta)+\log x=0.
\]
Assume that either
\begin{enumerate}
\item $l=2$, or
\item $l\ge3$ and $0<x\le x_{\rm cert}$.
\end{enumerate}
Equivalently, in the case $l\ge3$, the condition is
$\beta_*(x)\le L_{\rm cert}$.  Then
\[
  \frac1m\log Z(x)
  \longrightarrow \Phi_{d,l}(\beta_*(x))+\beta_*(x)\log x
\]
in probability.
\end{enumerate}

\end{theorem}

\medskip\noindent\textbf{Comment on the certified threshold.}
The notation above is only a compression of the two verified density intervals
proved later in the paper.  The $L_1$ interval comes from the star-shaped
analysis in Section~\ref{sm}, while the $L_2$ interval comes from the
one-dimensional reduction in Section~\ref{ptc}.  Since $\Phi'_{d,l}$ is strictly
decreasing, the condition $x\le x_{\rm cert}$ is equivalent to
$\beta_*(x)\le L_{\rm cert}$.

The first of these two explicit intervals has an especially simple form.  Indeed,
for $l\ge3$,
\[
  \frac{1}{dl-d-l+2}\le \frac1{1+\sqrt{(d-1)(l-1)}},
\]
so the second entry in the earlier definition of $L_1$ is redundant.  Moreover
\begin{equation}\label{eq:intro-L1-algebraic}
  \Phi'_{d,l}(L_1)\le0
  \quad\Longleftrightarrow\quad
  (d-1)(l-2)^l\le (l-1)^{l-1}.
\end{equation}
Equivalently, the $L_1$ part of the weighted theorem is
\[
  0<x\le \exp\{-\Phi'_{d,l}(L_1)\}
  =\frac{(l-1)^{l-1}}{(d-1)(l-2)^l}.
\]
The threshold $L_{\rm cert}$, or equivalently $x_{\rm cert}$, is not claimed to be
sharp; it is the largest density interval certified by the two closed-form criteria
$L_1$ and $L_2$ proved here.
\medskip

The next result gives another explicit criterion under which the same
replica-symmetric free-energy formula holds.  For $0<\beta<1/d$ and
$0<s<\beta$, set
\begin{equation}\label{dt}
  t_\beta(s):=\frac{(\beta-s)^{1/l}s^{1-2/l}}
  {(1-\beta-s)^{(l-1)/l}} .
\end{equation}
Also define
\begin{equation}\label{dc0}
  c_0(\beta):=\frac{3\beta l-2l+1+
  \sqrt{9\beta^2l^2-8\beta^2l-12\beta l^2+10\beta l+4l^2-4l+1}}
  {2l}.
\end{equation}
As recorded in Section~\ref{sm}, this is the zero in $(0,\beta)$ of the
one-variable function in \eqref{eq:c0-as-zero}; in particular, the interval in
Theorem~\ref{th12} is well defined.  Define
\begin{align}
G_\beta(s)&:=(1-(d-1)t_\beta(s))(d+l-dl)s^2 \notag\\
&\quad +2\beta(1-\beta d)(1-\beta)
        (1-(d-1)(l-1)t_\beta(s)) \label{dgb}\\
&\quad +s\Big[-t_\beta(s)(d-1)
        (\beta l-3\beta d-2l+2\beta^2d+3\beta dl \notag\\
&\hspace{38mm}{}-2\beta^2dl+1)
        +2\beta^2d-\beta l-\beta d+\beta dl-1\Big].\notag
\end{align}

\begin{theorem}\label{th12}
Let $\beta_*$ be as in Theorem~\ref{thm}.  If
\[
  G_{\beta_*}(s)>0,
  \qquad 0<s<\beta_*-c_0(\beta_*),
\]
then
\begin{equation}\label{czp}
  \frac1m\log Z\longrightarrow \Phi_{d,l}(\beta_*)
\end{equation}
in probability.
\end{theorem}

For orientation, we record a few concrete parameter values.

\begin{example}\label{ex:main-examples}
\begin{enumerate}
\item By \eqref{eq:intro-L1-algebraic}, the $L_1$ part of the certified condition
with $x=1$ holds precisely when
\[
  (d-1)(l-2)^l\le (l-1)^{l-1}.
\]
Thus Theorem~\ref{thm}(2) applies, for example, for all $l=2$, for
$l=3$ and $2\le d\le5$, and for $(d,l)=(2,4),(2,5)$.
For these parameters,
\[
  \frac1m\log Z\longrightarrow \Phi_{d,l}(\beta_*)
\]
in probability.

\item The condition in Theorem~\ref{th12} is a one-dimensional positivity
condition for an explicit function on the interval
$(0,\beta_*-c_0(\beta_*))$.  For any fixed pair $(d,l)$, it can be checked in a
certified way, for example by interval arithmetic applied to the explicit
function $G_{\beta_*}$.  No numerical verification is used in the proofs of the
theorems.
\end{enumerate}
\end{example}

For $d=2$, Theorem~\ref{thm}(2) gives exact limiting free energies for the
independent-set model on random $l$-regular graphs at activity $1$ for the
parameter values listed in Example~\ref{ex:main-examples}(1).  This is
complementary to \cite{LO18}, which studies independent sets on regular graphs
with vertex weight $\lambda>1$ and obtains upper bounds for the normalized free
energy in a large-activity regime.

The paper is organized as follows.  Section~\ref{fm} computes the first moment of
the number of matchings at a fixed density and the first moment of the total
number of matchings.  Section~\ref{sm} computes the fixed-density second moment
and proves explicit criteria for the replica-symmetric saddle to be the global
maximizer.  Section~\ref{sc} proves the required cycle estimates and the
conditional Poisson limits used in subgraph conditioning.  Section~\ref{fe}
proves convergence in probability of the unweighted free energy in the $L_1$
regime.  Section~\ref{wfe} proves the weighted free-energy statement in the
corresponding $L_1$ regime.  Section~\ref{ptc} gives the additional $L_2$
criterion, completes the proof of the certified-threshold formulation in
Theorem~\ref{thm}, and proves Theorem~\ref{th12}.
Section~\ref{mm} records a first-moment consequence for the upper tail of the
maximum matching size.  The appendix discusses the range of applicability of the
criterion in Theorem~\ref{th12}.

\section{First moment}\label{fm}

Throughout this section, $d$ and $l$ are fixed, and $m\to\infty$ through admissible values for which $lm/d\in\mathbb{N}$.  For an integer $h\geq 0$, let $Z_h$ denote the number of matchings consisting of exactly $h$ hyperedges in the random $(d,l)$-regular hypergraph $\mathcal{G}_{m,d,l}$.  When $h=m\beta$, we write $Z_{m\beta}$ for $Z_h$; all such asymptotics are understood along subsequences for which $m\beta\in\mathbb{N}$.  In this section we compute the first moment of $Z_{m\beta}$ and then the first moment of the total number of matchings $Z=\sum_h Z_h$.

Since a matching of size $h$ uses $lh$ distinct vertices and $\mathcal{G}_{m,d,l}$ has $lm/d$ vertices, $Z_h=0$ whenever $h>m/d$.  Assume below that
\begin{equation}
0<\beta<\frac1d, \qquad h=m\beta .
\label{cd4}
\end{equation}
Choose first the $h$ hyperedges that are to be present in the matching.  These hyperedges form a matching if and only if their $lh$ half-edges are incident to $lh$ distinct vertices.  Exposing the vertex-classes of these $lh$ half-edges sequentially gives
\begin{align}
\mathbb{E}Z_h
&=\binom{m}{h}\prod_{j=0}^{d-2}\prod_{i=0}^{lh-1}
\frac{l(m-h)-(d-1)i-j}{lm-1-di-j}
\notag\\
&=\frac{m!}{h!(m-h)!}\,
  \frac{(l(m-h))!}{(l(m-dh))!}\,
  \frac{(l(m-dh))!}{(lm)!}\,
  \frac{d^{lh}(lm/d)!}{(lm/d-lh)!}  \notag\\
&=\frac{m!}{h!(m-h)!}\,
  \frac{(l(m-h))!}{(lm)!}\,
  \frac{d^{lh}(lm/d)!}{(lm/d-lh)!} .
\label{ezb}
\end{align}
Indeed, after $i$ selected half-edges have already been assigned to distinct vertices, the $d-1$ other half-edges incident to the next selected half-edge must be chosen from the half-edges outside the selected hyperedges and not already used in these exposed vertex-classes.  This gives the product in the first line of \eqref{ezb}; the factorial form follows by collecting the corresponding factors.

For fixed $d,l$ and $\beta$, Stirling's formula applied to \eqref{ezb} gives
\begin{equation}
\mathbb{E}Z_{m\beta}=m^{-1/2}e^{m\Phi_{d,l}(\beta)+O(1)},
\label{a1m}
\end{equation}
where $\Phi_{d,l}$ is defined in \eqref{pdb}.  Differentiating \eqref{pdb},
\begin{align}
\Phi_{d,l}'(\beta)
&=-\ln\beta-(l-1)\ln(1-\beta)+l\ln(1-d\beta),
\label{1dp}\\
\Phi_{d,l}''(\beta)
&=-\frac{1}{\beta}
  -\frac{ld-d\beta-(l-1)}{(1-\beta)(1-d\beta)}<0,
\qquad 0<\beta<\frac1d .
\label{2dp}
\end{align}

\begin{lemma}\label{lm21}
The following hold.
\begin{enumerate}
\item $\Phi_{d,l}$ has a unique maximizer $\beta_*\in(0,1/d)$, and $\beta_*$ is the solution of \eqref{eq1}.  In particular,
\[
\max_{0<\beta<1/d}\Phi_{d,l}(\beta)=\Phi_{d,l}(\beta_*).
\]
\item If
\[
\frac{(l-1)(d-1)}{d}\ln\left(\frac{d-1}{d}\right)
-\frac1d\ln\left(\frac1d\right)
=:f_l\left(\frac1d\right)\geq 0,
\]
then $\Phi_{d,l}(\beta)>0$ for every $\beta\in(0,1/d)$.
\item If $f_l(1/d)<0$, then there is a unique $\beta_0\in(\beta_*,1/d)$ such that
\[
\Phi_{d,l}(\beta_0)=0.
\]
Moreover, $\Phi_{d,l}(\beta)>0$ for $\beta\in(0,\beta_0)$, and $\Phi_{d,l}(\beta)<0$ for $\beta\in(\beta_0,1/d)$.
\end{enumerate}
\end{lemma}

\begin{proof}
By \eqref{1dp},
\[
\lim_{\beta\downarrow0}\Phi_{d,l}'(\beta)=+\infty,
\qquad
\lim_{\beta\uparrow1/d}\Phi_{d,l}'(\beta)=-\infty.
\]
Together with the strict concavity in \eqref{2dp}, this implies that $\Phi_{d,l}'$ has a unique zero in $(0,1/d)$.  The equation $\Phi_{d,l}'(\beta)=0$ is equivalent to
\[
(1-d\beta)^l=\beta(1-\beta)^{l-1},
\]
which is \eqref{eq1}.  Hence this unique zero is $\beta_*$, and it is the unique maximizer of $\Phi_{d,l}$.

The continuous extension of $\Phi_{d,l}$ to the endpoints satisfies
\[
\lim_{\beta\downarrow0}\Phi_{d,l}(\beta)=0
\]
and
\[
\lim_{\beta\uparrow1/d}\Phi_{d,l}(\beta)
=\frac{(l-1)(d-1)}{d}\ln\left(\frac{d-1}{d}\right)
-\frac1d\ln\left(\frac1d\right)
=f_l\left(\frac1d\right).
\]
If $f_l(1/d)\geq0$, strict concavity implies that $\Phi_{d,l}$ lies strictly above the chord joining its endpoint values; hence $\Phi_{d,l}(\beta)>0$ for all $\beta\in(0,1/d)$.

If $f_l(1/d)<0$, then $\Phi_{d,l}(\beta_*)>0$, because $\Phi_{d,l}(0+)=0$ and $\Phi_{d,l}$ initially increases.  Since $\Phi_{d,l}$ is strictly decreasing on $(\beta_*,1/d)$, there is a unique zero $\beta_0\in(\beta_*,1/d)$.  The asserted signs follow from the monotonicity on $(0,\beta_*)$ and $(\beta_*,1/d)$.
\end{proof}

\begin{lemma}\label{l12}
Let $n\in\mathbb{N}$.  Then
\begin{equation}
\sqrt{2\pi n}\left(\frac{n}{e}\right)^n e^{\frac{1}{12n+1}}
\leq n!\leq
\sqrt{2\pi n}\left(\frac{n}{e}\right)^n e^{\frac{1}{12n}}.
\label{bnf}
\end{equation}
Moreover,
\[
\sqrt{2\pi}\, n^{n+1/2}e^{-n}\leq n!\leq e\, n^{n+1/2}e^{-n}.
\]
\end{lemma}

\begin{proof}
See Exercise 3.1.9 of \cite{GS97}.
\end{proof}

\begin{lemma}\label{lem:first-moment-local}
Let $K\subset(0,1/d)$ be compact.  Then
\[
\sup_{\substack{h\in\mathbb{N}:\ h/m\in K}}
\left|
\sqrt{m}\,e^{-m\Phi_{d,l}(h/m)}\mathbb{E}Z_h
-\frac{1}{\sqrt{2\pi(h/m)(1-dh/m)}}
\right|\longrightarrow 0.
\]
In particular, if $\beta\in(0,1/d)$ is fixed and $m\beta\in\mathbb{N}$, then
\[
\lim_{m\to\infty}\sqrt{m}\,e^{-m\Phi_{d,l}(\beta)}\mathbb{E}Z_{m\beta}
=\frac{1}{\sqrt{2\pi\beta(1-d\beta)}}.
\]
\end{lemma}

\begin{proof}
By \eqref{ezb} and \eqref{bnf}, for $1\leq h<m/d$,
\begin{equation}
\frac{B(m,h)e^{m\Phi_{d,l}(h/m)}}
{\sqrt{2m\pi(h/m)(1-dh/m)}}
\leq \mathbb{E}Z_h\leq
\frac{A(m,h)e^{m\Phi_{d,l}(h/m)}}
{\sqrt{2m\pi(h/m)(1-dh/m)}},
\label{azh}
\end{equation}
where
\begin{align*}
\log A(m,h)
&=\frac{1}{12m}+\frac{1}{12l(m-h)}+\frac{d}{12lm}
  -\frac{1}{12h+1}-\frac{1}{12(m-h)+1}\\
&\hspace{2em}
  -\frac{1}{12lm+1}-\frac{d}{12l(m-dh)+d},\\
\log B(m,h)
&=\frac{1}{12m+1}+\frac{1}{12l(m-h)+1}+\frac{d}{12lm+d}
  -\frac{1}{12h}-\frac{1}{12(m-h)}\\
&\hspace{2em}
  -\frac{1}{12lm}-\frac{d}{12l(m-dh)}.
\end{align*}
If $h/m$ stays in a compact subset of $(0,1/d)$, then all denominators in the exponents above are of order $m$, uniformly in $h$.  Hence $A(m,h)\to1$ and $B(m,h)\to1$ uniformly on such compact subsets, and the lemma follows from \eqref{azh}.
\end{proof}

\noindent\textbf{Proof of Theorem \ref{thm}(1).}
Let $H_m=\lfloor m/d\rfloor$.  Since $Z=\sum_{h=0}^{H_m}Z_h$, it remains to evaluate the lattice sum of the estimates in Lemma \ref{lem:first-moment-local} near the unique maximizer $\beta_*$.  Put
\[
a=-\Phi_{d,l}''(\beta_*)>0.
\]
Choose $\delta>0$ such that $[\beta_* -\delta,\beta_*+\delta]\subset(0,1/d)$.  By strict concavity and continuity of $\Phi_{d,l}''$, after decreasing $\delta$ if necessary,
\[
\Phi_{d,l}(\beta)\leq \Phi_{d,l}(\beta_*)-\frac{a}{4}(\beta-\beta_*)^2,
\qquad |\beta-\beta_*|\leq\delta .
\]
The crude form of Stirling's bounds in Lemma \ref{l12}, applied to \eqref{ezb}, gives
\[
\sup_{0\leq h\leq H_m}
\left|\frac1m\log \mathbb{E}Z_h-\Phi_{d,l}\left(\frac{h}{m}\right)\right|=o(1),
\]
where $\Phi_{d,l}$ is understood through its continuous endpoint limits.  On the complement of $[\beta_* -\delta,\beta_*+\delta]$ in $[0,1/d]$, this continuous extension is bounded above by $\Phi_{d,l}(\beta_*)-\eta$ for some $\eta>0$.  Hence the contribution to $e^{-m\Phi_{d,l}(\beta_*)}\mathbb{E}Z$ from $|h/m-\beta_*|>\delta$ is exponentially small.  The contribution from $K/\sqrt{m}<|h/m-\beta_*|\leq\delta$ is bounded by a Gaussian tail and tends to $0$ as $K\to\infty$, uniformly for large $m$.

It is therefore enough to sum over $|h-m\beta_*|\leq K\sqrt{m}$ and then let $K\to\infty$.  For such $h$, Taylor's theorem gives, uniformly in $h$,
\[
\Phi_{d,l}\left(\frac{h}{m}\right)-\Phi_{d,l}(\beta_*)
=\frac{\Phi_{d,l}''(\beta_*)}{2}\left(\frac{h}{m}-\beta_*\right)^2
+O\left(\left|\frac{h}{m}-\beta_*\right|^3\right).
\]
Using Lemma \ref{lem:first-moment-local}, with $x_h=(h-m\beta_*)/\sqrt{m}$, we obtain
\begin{align*}
\lim_{m\to\infty} e^{-m\Phi_{d,l}(\beta_*)}
\sum_{|h-m\beta_*|\leq K\sqrt{m}}\mathbb{E}Z_h
&=\frac{1}{\sqrt{2\pi\beta_*(1-d\beta_*)}}
\int_{-K}^{K} e^{-a x^2/2}\,dx .
\end{align*}
Letting $K\to\infty$ yields
\[
\lim_{m\to\infty}e^{-m\Phi_{d,l}(\beta_*)}\mathbb{E}Z
=\frac{1}{\sqrt{\beta_*(1-d\beta_*)[-\Phi_{d,l}''(\beta_*)]}}.
\]
Finally, by \eqref{2dp},
\[
-\Phi_{d,l}''(\beta_*)
=\frac{1+(ld-d-l)\beta_*}{\beta_*(1-\beta_*)(1-d\beta_*)}.
\]
Thus
\[
\lim_{m\to\infty}e^{-m\Phi_{d,l}(\beta_*)}\mathbb{E}Z
=\sqrt{\frac{1-\beta_*}{1+(ld-d-l)\beta_*}},
\]
as claimed.
\hfill$\Box$

\section{Second moment}\label{sm}

Throughout this section $d,l\geq 2$ are fixed and $lm/d\in\mathbb N$.  We write
$Z_{m\beta}$ for the number of matchings with exactly $m\beta$ hyperedges, and all
statements involving $Z_{m\beta}$ are understood along subsequences for which
$m\beta\in\mathbb N$.

The purpose of this section is to compute the second moment of $Z_{m\beta}$ and to
identify a regime in which the second moment has the same exponential growth rate as
$(\mathbb E Z_{m\beta})^2$.

Let $(M_1,M_2)$ be an ordered pair of matchings, each of size $m\beta$.  Put
\[
A=M_1\cap M_2,
\qquad B=M_1\setminus M_2,
\qquad C=M_2\setminus M_1,
\qquad D=[m]\setminus(M_1\cup M_2),
\]
and write $s=|A|$.  Thus $|B|=|C|=m\beta-s$ and
$|D|=m-2m\beta+s$.  Let $t$ be the number of vertices that are incident to one
half-edge belonging to a hyperedge of $B$ and one half-edge belonging to a hyperedge of
$C$.  Since $M_1$ and $M_2$ are matchings, such a vertex is incident to exactly one
half-edge from $B$ and exactly one half-edge from $C$ among the half-edges in
$M_1\cup M_2$.  Set
\begin{equation}
\rho:=\frac{s}{m},\qquad \theta:=\frac{t}{m}.\label{drt}
\end{equation}

\begin{lemma}\label{lem:second-feasible}
For every pair of matchings as above,
\begin{equation}
0\leq \rho\leq \beta,\qquad 0\leq \theta\leq l\beta-l\rho,\label{ccd1}
\end{equation}
and
\begin{equation}
\left(2\beta-\rho-\frac{\theta}{l}\right)d\leq 1.\label{cd2}
\end{equation}
Equivalently,
\begin{equation*}
\max\left\{0,2lm\beta-ls-\frac{lm}{d}\right\}\leq t\leq lm\beta-ls .
\end{equation*}
\end{lemma}

\begin{proof}
The bounds in \eqref{ccd1} are immediate from the definitions.  It remains to prove
\eqref{cd2}.  There are
\[
ls+2\{l(m\beta-s)-t\}=lm(2\beta-\rho)-2\theta m
\]
half-edges of type $A$, together with those half-edges of $B\cup C$ that do not collide
with a half-edge from the other matching.  No two of these half-edges are incident to the
same vertex.  The vertices containing them therefore account for
$(lm(2\beta-\rho)-2\theta m)d$ half-edges.  The $2t$ colliding half-edges from
$B\cup C$ form $t$ vertices, which account for a further $\theta md$ half-edges, and
these vertices are disjoint from the previous ones.  Since the hypergraph has $lm$
half-edges in total,
\[
(lm(2\beta-\rho)-2\theta m)d+\theta md\leq lm,
\]
which is \eqref{cd2}.
\end{proof}

Let
\begin{equation}
\mathcal R_\beta:=\left\{(\rho,\theta):0<\rho<\beta,
\quad \max\left\{0,2l\beta-l\rho-\frac{l}{d}\right\}<\theta<l\beta-l\rho\right\}.
\label{drb}
\end{equation}

\begin{lemma}[Exact second moment]\label{lem:second-exact}
For $0<\beta<1/d$ and $m\beta\in\mathbb N$,
\begin{equation}
\mathbb E Z_{m\beta}^2=
\sum_{s=0}^{m\beta}\sum_{t=\max\left\{0,2lm\beta-ls-\frac{lm}{d}\right\}}^{lm\beta-ls}
F_\beta(s,t),\label{2ndm}
\end{equation}
where
\begin{equation}\label{dfb}
\begin{aligned}
F_\beta(s,t)
&=\frac{m!(d-1)^t}{s!(m\beta-s)!^2(m-2m\beta+s)!}
\frac{[l(m\beta-s)]!^2}{t![l(m\beta-s)-t]!^2}\\
&\quad \times
\frac{[l(m-2m\beta+s)]!}{(lm)!}
\frac{d^{2lm\beta-ls-t}\left(\frac{lm}{d}\right)!}
{\left(\frac{lm}{d}-2lm\beta+ls+t\right)!}.
\end{aligned}
\end{equation}
Here factorials with integer arguments are interpreted only for admissible pairs $(s,t)$.
\end{lemma}

\begin{proof}
First choose the two edge sets.  The number of ordered pairs $(M_1,M_2)$ with
$|M_1|=|M_2|=m\beta$ and $|M_1\cap M_2|=s$ is
\[
\binom ms\binom{m-s}{m\beta-s}\binom{m-m\beta}{m\beta-s}
=\frac{m!}{s!(m\beta-s)!^2(m-2m\beta+s)!}.
\]
After the edge sets have been chosen, choose the $t$ half-edges of $B$ and the $t$
half-edges of $C$ that collide, and pair them.  This gives
\[
\binom{l(m\beta-s)}t^2t!
=\frac{[l(m\beta-s)]!^2}{t![l(m\beta-s)-t]!^2}
\]
choices, and the factor $(d-1)^t$ chooses the remaining half-edges at the $t$ vertices
where one $B$-half-edge and one $C$-half-edge collide.  The remaining constrained
vertices containing half-edges of $M_1\cup M_2$ are then completed using half-edges from
hyperedges in $D$, and the still-unpaired half-edges are partitioned into vertices.  Dividing
by the total number of configurations gives the final factor
\[
\frac{[l(m-2m\beta+s)]!}{(lm)!}
\frac{d^{2lm\beta-ls-t}\left(\frac{lm}{d}\right)!}
{\left(\frac{lm}{d}-2lm\beta+ls+t\right)!}.
\]
Multiplying the displayed factors gives \eqref{dfb}; summing over the admissible $s,t$
gives \eqref{2ndm}.
\end{proof}

For $(\rho,\theta)\in\mathcal R_\beta$, define
\begin{align}
\Psi_{d,l}(\beta,\rho,\theta)
&=-\rho\log\rho-\theta\log\frac{\theta}{l}
+(l-1)(1-2\beta+\rho)\log(1-2\beta+\rho)+\theta\log(d-1)\notag\\
&\quad +2(l-1)(\beta-\rho)\log(\beta-\rho)
-2l\left(\beta-\rho-\frac{\theta}{l}\right)
\log\left(\beta-\rho-\frac{\theta}{l}\right)\notag\\
&\quad -\frac{l}{d}\left(1-2\beta d+\rho d+\frac{\theta d}{l}\right)
\log\left(1-2\beta d+\rho d+\frac{\theta d}{l}\right).
\label{pdp}
\end{align}
As usual, $x\log x$ is interpreted as $0$ at $x=0$ when boundary values are considered.

\begin{lemma}[Uniform Stirling estimate]\label{lem:second-stirling}
Let $K\subset\mathcal R_\beta$ be compact.  Then, uniformly for lattice points
with $(s/m,t/m)\in K$,
\begin{align}
F_\beta(s,t)
&=m^{-2}e^{m\Psi_{d,l}(\beta,s/m,t/m)}
\frac{1+O_K(m^{-1})}
{4\pi^2\left(\beta-\frac{s}{m}-\frac{t}{lm}\right)
\sqrt{\frac{s}{m}\frac{t}{m}
\left(1-2\beta d+\frac{sd}{m}+\frac{td}{lm}\right)}}.
\label{eq:second-stirling}
\end{align}
\end{lemma}

\begin{proof}
This follows by applying the uniform Stirling bounds in Lemma \ref{l12} to the factorial
expression \eqref{dfb}.  The compactness assumption keeps all arguments of the logarithms
and all denominator factors bounded away from zero.
\end{proof}

The first derivatives are
\begin{align*}
\partial_\rho\Psi_{d,l}
&=-\log\rho+(l-1)\log(1-2\beta+\rho)-2(l-1)\log(\beta-\rho)\\
&\quad +2l\log\left(\beta-\rho-\frac{\theta}{l}\right)
-l\log\left(1-2\beta d+\rho d+\frac{\theta d}{l}\right),\\
\partial_\theta\Psi_{d,l}
&=-\log\frac{\theta}{l}+2\log\left(\beta-\rho-\frac{\theta}{l}\right)
-\log\left(1-2\beta d+\rho d+\frac{\theta d}{l}\right)+\log(d-1).
\end{align*}
The second derivatives are
\begin{align}
\partial_{\rho\rho}^2\Psi_{d,l}
&=-\frac1\rho+\frac{l-1}{1-2\beta+\rho}
+\frac{2(l-1)}{\beta-\rho}
-\frac{2l}{\beta-\rho-\theta/l}
-\frac{ld}{1-2\beta d+\rho d+\theta d/l},\label{2d1}\\
\partial_{\theta\theta}^2\Psi_{d,l}
&=-\frac1\theta-\frac{2/l}{\beta-\rho-\theta/l}
-\frac{d/l}{1-2\beta d+\rho d+\theta d/l}<0,\label{nsd}\\
\partial_{\rho\theta}^2\Psi_{d,l}
&=-\frac2{\beta-\rho-\theta/l}
-\frac d{1-2\beta d+\rho d+\theta d/l}.\label{eq:mixed-second}
\end{align}
Let
\[
H(\beta,\rho,\theta)=
\begin{pmatrix}
\partial_{\rho\rho}^2\Psi_{d,l} & \partial_{\rho\theta}^2\Psi_{d,l}\\
\partial_{\rho\theta}^2\Psi_{d,l} & \partial_{\theta\theta}^2\Psi_{d,l}
\end{pmatrix}(\beta,\rho,\theta).
\]

The critical point equations are equivalent to
\begin{align}
(d-1)\left(\beta-\rho-\frac{\theta}{l}\right)^2
&=\frac{\theta}{l}\left(1-2\beta d+\rho d+\frac{\theta d}{l}\right),\label{eq3}\\
\rho(d-1)^l(\beta-\rho)^{2(l-1)}
&=\left(\frac{\theta}{l}\right)^l(1-2\beta+\rho)^{l-1}.
\label{eq2}
\end{align}
The point
\begin{equation}
(\rho_\beta,\theta_\beta):=(\beta^2,l(d-1)\beta^2)
\label{sl1}
\end{equation}
satisfies \eqref{eq3}--\eqref{eq2}, belongs to $\mathcal R_\beta$, and
\begin{equation}
\Psi_{d,l}(\beta,\rho_\beta,\theta_\beta)=2\Phi_{d,l}(\beta).
\label{eq:Psi-at-saddle}
\end{equation}

By \eqref{eq2}, any critical point satisfies
\begin{equation}
\frac{\theta}{l}=
\frac{(d-1)\rho^{1/l}(\beta-\rho)^{2(l-1)/l}}
{(1-2\beta+\rho)^{(l-1)/l}}
=:(d-1)a(\rho).
\label{sol}
\end{equation}
Substituting this into \eqref{eq3} gives
\begin{equation*}
h(\rho):=\log a(\rho)+\log\{1-2\beta d+\rho d+d(d-1)a(\rho)\}
-2\log\{\beta-\rho-(d-1)a(\rho)\}=0.
\end{equation*}

For the next two lemmas we record a version of the critical-point argument which keeps
track of the domain on which the logarithm defining $h$ is meaningful.  Recall that, for
$0<s<\beta$,
\[
 t_\beta(s):=
 \frac{(\beta-s)^{1/l}s^{1-2/l}}{(1-\beta-s)^{(l-1)/l}},
\]
and
\begin{align*}
G_\beta(s):={}&(1-(d-1)t_\beta(s))(d+l-dl)s^2
+2\beta(1-\beta d)(1-\beta)(1-(d-1)(l-1)t_\beta(s))\\
&+s\Big[-t_\beta(s)(d-1)\big(\beta l-3\beta d-2l+2\beta^2d+3\beta dl-2\beta^2dl+1\big)\\
&\hspace{35mm}+2\beta^2d-\beta l-\beta d+\beta dl-1\Big].
\end{align*}
Also let $c_0=c_0(\beta)$ be the unique zero in $(0,\beta)$ of
\begin{equation}
\frac{a'(\rho)}{a(\rho)}=\frac1l\left\{\frac1\rho-
\frac{2(l-1)}{\beta-\rho}-\frac{l-1}{1-2\beta+\rho}\right\};
\label{eq:c0-as-zero}
\end{equation}
equivalently,
\[
 c_0=\frac{3\beta l-2l+1+
 \sqrt{9\beta^2l^2-8\beta^2l-12\beta l^2+10\beta l+4l^2-4l+1}}{2l}.
\]

\begin{lemma}[Uniqueness of the admissible critical point]\label{le33}
Assume $0<\beta<1/d$ and that either
\[
 \beta<\frac{4-\sqrt2}{7},
 \qquad\text{or}\qquad d\ge3 .
\]
Assume further that
\begin{equation}
G_\beta(s)>0,\qquad 0<s<\beta-c_0 .
\label{eq:G-positive}
\end{equation}
Then $\Psi_{d,l}(\beta,\cdot,\cdot)$ has exactly one critical point in
$\mathcal R_\beta$, namely
\[
(\rho,\theta)=(\beta^2,l(d-1)\beta^2).
\]
Equivalently, the equation $h(\rho)=0$ has exactly one solution in the admissible domain
where the corresponding point $(\rho,l(d-1)a(\rho))$ lies in $\mathcal R_\beta$.
\end{lemma}

\begin{proof}
Let
\[
 R(\rho):=\beta-\rho-(d-1)a(\rho),
 \qquad
 S(\rho):=1-2\beta d+\rho d+d(d-1)a(\rho).
\]
After the substitution \eqref{sol}, the point $(\rho,l(d-1)a(\rho))$ belongs to
$\mathcal R_\beta$ if and only if
\begin{equation}
\rho\in\mathcal I_\beta:=\{\rho\in(0,\beta): R(\rho)>0,\; S(\rho)>0\}.
\label{eq:Ibeta}
\end{equation}
On this set $h$ is well-defined and
\[
 h(\rho)=\log a(\rho)+\log S(\rho)-2\log R(\rho).
\]
Notice the useful identity
\begin{equation}
S(\rho)=1-\beta d-dR(\rho),
\label{eq:S-R-identity}
\end{equation}
so that
\[
\mathcal I_\beta=\left\{\rho\in(0,\beta):
0<R(\rho)<\frac{1-\beta d}{d}\right\}.
\]
Moreover, $a(\beta^2)=\beta^2$, and therefore
\[
R(\beta^2)=\beta(1-d\beta)>0,
\qquad
S(\beta^2)=(1-d\beta)^2>0.
\]
Thus $\beta^2\in\mathcal I_\beta$, and direct substitution gives $h(\beta^2)=0$.

We next prove that $h$ is strictly increasing on every connected component of
$\mathcal I_\beta$.  Put
\[
A(\rho):=\frac{a'(\rho)}{a(\rho)}
=\frac1l\left\{\frac1\rho-
\frac{2(l-1)}{\beta-\rho}-\frac{l-1}{1-2\beta+\rho}\right\}.
\]
Under the stated assumption on $\beta$ we have
$\beta-\rho<\sqrt2(1-2\beta+\rho)$ for every $\rho\in(0,\beta)$; when $d\ge3$ this follows
from $\beta<1/d\le1/3<(4-\sqrt2)/7$.  Hence
\[
\frac{d}{d\rho}A(\rho)
=-\frac1l\left\{\frac1{\rho^2}+\frac{2(l-1)}{(\beta-\rho)^2}
-\frac{l-1}{(1-2\beta+\rho)^2}\right\}<0.
\]
Since $A(\rho)\to+\infty$ as $\rho\downarrow0$ and
$A(\rho)\to-\infty$ as $\rho\uparrow\beta$, the point $c_0$ is the unique zero of
$A$; in particular $A(\rho)>0$ on $(0,c_0)$ and $A(\rho)<0$ on $(c_0,\beta)$.

Differentiating $h$ gives, on $\mathcal I_\beta$,
\begin{align*}
 h'(\rho)
&=\frac{d}{S(\rho)}+\frac{2}{R(\rho)}
+A(\rho)\left(1+\frac{d(d-1)a(\rho)}{S(\rho)}
+\frac{2(d-1)a(\rho)}{R(\rho)}\right).
\end{align*}
All factors $R,S,a$ are positive on $\mathcal I_\beta$.  Therefore $h'(\rho)>0$ whenever
$\rho\le c_0$.  If $\rho>c_0$, set $s=\beta-\rho$.  Since
$t_\beta(s)=a(\rho)/(\beta-\rho)$, a direct simplification gives
\begin{equation}
 h'(\rho)=
 \frac{G_\beta(\beta-\rho)}
 {l\rho(1-2\beta+\rho)R(\rho)S(\rho)} .
\label{eq:hprime-G}
\end{equation}
The denominator in \eqref{eq:hprime-G} is positive on $\mathcal I_\beta$, and
\eqref{eq:G-positive} gives $G_\beta(\beta-\rho)>0$ for $\rho>c_0$.  Thus $h'>0$ on
$\mathcal I_\beta$.

It remains only to check that the admissible domain $\mathcal I_\beta$ has a single
connected component.  Let $J$ be a connected component of $\mathcal I_\beta$.  Since $h$
is strictly increasing on $J$, a finite left endpoint of $J$ cannot be a zero of $R$; indeed
$R\downarrow0$ would force $h\to+\infty$ from within $J$, which is incompatible with
monotonicity.  Similarly, a finite right endpoint of $J$ cannot be a zero of $S$, since
$S\downarrow0$ would force $h\to-\infty$ from within $J$.

Suppose that there were two distinct components $J_1<J_2$.  By the preceding paragraph,
the right endpoint of $J_1$ must be a zero of $R$, while the left endpoint of $J_2$ must be
a zero of $S$.  Using \eqref{eq:S-R-identity}, the latter means
$R=(1-\beta d)/d>0$.  Starting from a zero of $R$ at the right endpoint of $J_1$ and moving
towards the left endpoint of $J_2$, continuity would then produce a first point after which
$R>0$; near that point $S>0$ as well, so a component of $\mathcal I_\beta$ would have a
left endpoint at which $R=0$, contradicting the previous paragraph.  Hence
$\mathcal I_\beta$ is connected.

Since $\beta^2\in\mathcal I_\beta$, $h(\beta^2)=0$, and $h$ is strictly increasing on the
connected set $\mathcal I_\beta$, this is the unique admissible solution of $h(\rho)=0$.
Every critical point in $\mathcal R_\beta$ must satisfy \eqref{eq2}, hence \eqref{sol}, and
then $h(\rho)=0$; conversely, \eqref{eq2} and \eqref{eq3} are exactly the exponentiated
critical point equations.  The unique critical point is therefore
$(\beta^2,l(d-1)\beta^2)$.
\end{proof}

\begin{lemma}[$G_\beta$ criterion]\label{le34}
Under the assumptions of Lemma~\ref{le33}, the rate function has a unique
global maximizer on $\overline{\mathcal R_\beta}$.  The maximizer is
\[
(\rho,\theta)=(\beta^2,l(d-1)\beta^2),
\]
and the maximum value is $2\Phi_{d,l}(\beta)$.
\end{lemma}

\begin{proof}
Extend $\Psi_{d,l}$ continuously to the closure of $\mathcal R_\beta$ by the convention
$x\log x=0$ at $x=0$.  The closure is compact, so the extension has a maximizer.
We claim that no maximizer lies on the boundary.  On the boundary piece $\rho=0$, the
inward derivative $\partial_\rho\Psi$ tends to $+\infty$.  On the boundary piece $\theta=0$,
the inward derivative $\partial_\theta\Psi$ tends to $+\infty$.  On the upper boundary
$\beta-\rho-\theta/l=0$, moving slightly inward by decreasing $\theta$ increases
$\Psi$, since $\partial_\theta\Psi\to-\infty$.  On the lower boundary
$1-2\beta d+\rho d+\theta d/l=0$, moving slightly inward by increasing $\theta$ increases
$\Psi$, since $\partial_\theta\Psi\to+\infty$.  The same one-sided perturbations, or their
obvious combinations, handle the boundary corners.  Therefore any maximizer of the
continuous extension lies in the interior $\mathcal R_\beta$.

An interior maximizer is a critical point.  Lemma~\ref{le33} shows that the only critical
point in $\mathcal R_\beta$ is $(\beta^2,l(d-1)\beta^2)$, so this point is the unique global
maximizer.  The value at this point is \eqref{eq:Psi-at-saddle}.
\end{proof}

\begin{lemma}\label{lm33}
Assume $0<\beta<1/d$.  For every $(\rho,\theta)\in\mathcal R_\beta$,
\begin{equation}
\partial_{\rho\rho}^2\Psi_{d,l}(\beta,\rho,\theta)<0.
\label{n2d}
\end{equation}
\end{lemma}

\begin{proof}
Let $S=1-2\beta d+\rho d+\theta d/l$.  Since $(\rho,\theta)\in\mathcal R_\beta$,
$S>0$ and $\theta<l(\beta-\rho)<l/d$.  Therefore
\begin{align*}
\frac{l-1}{1-2\beta+\rho}-\frac{ld}{S}
&=\frac{-l(d-1)+\theta d(l-1)/l-(1-2\beta d+\rho d)}{(1-2\beta+\rho)S}\\
&\leq \frac{-l(d-2)-S}{(1-2\beta+\rho)S}<0.
\end{align*}
Also
\[
\frac{2(l-1)}{\beta-\rho}-\frac{2l}{\beta-\rho-\theta/l}<0.
\]
Combining these two inequalities with \eqref{2d1} proves \eqref{n2d}.
\end{proof}

\begin{lemma}\label{l24}
Assume $0<\beta<1/d$.
\begin{enumerate}
\item If
\begin{equation}
l\leq d,
\label{ld2}
\end{equation}
then the point \eqref{sl1} is a local maximizer for every $0<\beta<1/d$.
\item If
\begin{equation}
\beta<\frac1{1+\sqrt{(d-1)(l-1)}},
\label{bld}
\end{equation}
then the point \eqref{sl1} is a local maximizer.
\end{enumerate}
\end{lemma}

\begin{proof}
At \eqref{sl1},
\begin{align}
\det H(\beta,\beta^2,l(d-1)\beta^2)
=\frac{1-2\beta+\beta^2(d+l-dl)}
{l\beta^4(1-d\beta)^2(1-\beta)^2(d-1)}.
\label{dth}
\end{align}
The numerator can be written as
\[
(1-\beta)^2-(d-1)(l-1)\beta^2.
\]
Thus \eqref{bld} is exactly the condition that \eqref{dth} be positive.  Together with
\eqref{n2d} and \eqref{nsd}, positivity of the determinant implies that the Hessian is
negative definite, proving part (2).  If $l\leq d$ and $0<\beta<1/d$, then
\eqref{bld} holds; this gives part (1).
\end{proof}

\begin{lemma}\label{l04}
If $l=2$, then for every $0<\beta<1/d$ the Hessian matrix
$H(\beta,\rho,\theta)$ is negative definite throughout $\mathcal R_\beta$.
Consequently, \eqref{sl1} is the unique global maximizer of
$\Psi_{d,2}(\beta,\cdot,\cdot)$ on $\mathcal R_\beta$.
\end{lemma}

\begin{proof}
By \eqref{n2d} and \eqref{nsd}, it suffices to prove $\det H>0$.  Put
\begin{align*}
A&=-\frac1\rho+\frac1{1-2\beta+\rho}+\frac2{\beta-\rho},\\
B&=-\frac2{\beta-\rho-\theta/2}-\frac d{1-2\beta d+\rho d+\theta d/2},\\
C&=-\frac1\theta,
\qquad
U=-\frac1\rho+\frac1{1-2\beta+\rho}.
\end{align*}
Then $B,C,U<0$ and
\[
\det H=(A+2B)\left(C+\frac B2\right)-B^2
=AC+B\left(\frac A2+2C\right).
\]
Since $AC=CU-2/[\theta(\beta-\rho)]\geq -2/[\theta(\beta-\rho)]$, and since
$\theta<2(\beta-\rho)$,
\begin{align*}
B\left(\frac A2+2C\right)
&=\frac{BU}{2}+B\left(\frac1{\beta-\rho}-\frac2\theta\right)\\
&\geq \frac4{(\beta-\rho)\theta},
\end{align*}
we get $\det H>0$.  Strict concavity on the convex set $\mathcal R_\beta$, together
with the critical point equations, gives the unique global maximizer.
\end{proof}

\subsection*{\texorpdfstring{The global maximum for $l\geq3$ in the small-density regime}{The global maximum for l >= 3 in the small-density regime}}

In this subsection set $\eta=\theta/l$ and write
\[
\widehat\Psi_\beta(\rho,\eta):=\Psi_{d,l}(\beta,\rho,l\eta),
\qquad
\widehat{\mathcal R}_\beta:=\{(\rho,\eta):(\rho,l\eta)\in\mathcal R_\beta\}.
\]
Thus
\[
\widehat{\mathcal R}_\beta=
\left\{(\rho,\eta):0<\rho<\beta,
\quad b_-(\rho)<\eta<b_+(\rho)\right\},
\]
where
\[
b_-(\rho):=\max\left\{0,2\beta-\rho-\frac1d\right\},
\qquad b_+(\rho):=\beta-\rho.
\]
The determinant of the Hessian with respect to $(\rho,\eta)$ is $l^2\det H(\beta,\rho,l\eta)$, so the sign of the determinant is unchanged by this change of variables.

For $l\geq3$, define
\begin{equation}
\rho_5:=\frac{\beta(1-2\beta)}{l(1-\beta)-1},
\qquad
\rho_3:=\frac{\beta(1-2\beta)}{\sqrt{2l}(1-\beta)-1}.
\label{dr5dr3}
\end{equation}
Under $0<\beta<1/d$ these quantities are positive, and $\rho_5<\rho_3<\beta$.

\begin{lemma}[The curve $\det H=0$]\label{l26}
Assume $d\geq2$, $l\geq3$, and $0<\beta<1/d$.  For $(\rho,\eta)\in\widehat{\mathcal R}_\beta$,
\begin{equation*}
\det H(\beta,\rho,l\eta)=0
\end{equation*}
if and only if
\begin{equation}
\eta=\xi(\rho):=(\beta-\rho)\frac{J(\rho)}{D(\rho)},
\label{dxi-def}
\end{equation}
where
\begin{align}
D(\rho)&=(dl-l-d)\rho^2
+(3\beta l-\beta d-2l+2\beta^2d+\beta dl-2\beta^2dl+1)\rho-2\beta^2+\beta,
\label{dfd}\\
J(\rho)&=(dl-l-d)\rho^2
+(\beta d+\beta l+2\beta^2d-\beta dl-1)\rho
+2\beta^2d-\beta-4\beta^3d+2\beta^2.
\label{dj}
\end{align}
Moreover,
\begin{equation}
\xi'(\rho)=-1+\frac{F(\rho)}{D(\rho)^2},
\label{dxi}
\end{equation}
where
\begin{equation}
F(\rho)=2(\beta d-1)^2(l-1)
\left[(2l(1-\beta)^2-1)\rho^2
-2\beta(1-2\beta)\rho-\beta^2(1-2\beta)^2\right].
\label{dF}
\end{equation}
The equation $F(\rho)=0$ has the unique positive root $\rho_3$ in \eqref{dr5dr3};
$F(\rho)<0$ for $0<\rho<\rho_3$ and $F(\rho)>0$ for $\rho>\rho_3$.
\end{lemma}

\begin{proof}
Substituting the second derivatives \eqref{2d1}--\eqref{eq:mixed-second} into
$\det H(\beta,\rho,l\eta)=0$ and solving for $\eta$ gives \eqref{dxi-def}.  Differentiating
\eqref{dxi-def} gives \eqref{dxi}.  The asserted formula for the unique positive root of
$F$ follows from the displayed quadratic in \eqref{dF}.
\end{proof}

\begin{lemma}[Monotonicity of the determinant curve]\label{l39}
Assume $d\geq2$, $l\geq3$, and $0<\beta<1/d$.  Then $\xi'$ is strictly increasing on
$(\rho_5,\beta)$.
\end{lemma}

\begin{proof}
By \eqref{dxi},
\begin{equation}
\xi''(\rho)=\frac{F'(\rho)D(\rho)-2F(\rho)D'(\rho)}{D(\rho)^3}
=\frac{4(1-\beta d)^2(l-1)K(\rho)}{D(\rho)^3},
\label{eq:xi-second}
\end{equation}
where $K$ is the cubic polynomial defined by the second equality.  A direct calculation gives
\begin{equation}
K'(\rho)=3(dl-d-l)L(\rho),
\qquad
L(\rho):=-2l(1-\beta)^2\rho^2+(\rho+\beta(1-2\beta))^2.
\label{dkr}
\end{equation}
The positive root of $L$ is $\rho_3$, and $L(\rho)>0$ on $(\rho_5,\rho_3)$ while
$L(\rho)<0$ on $(\rho_3,\beta)$.  Hence $K$ is increasing on $(\rho_5,\rho_3)$ and
decreasing on $(\rho_3,\beta)$.  Since $F(\rho_3)=0$ and $F'(\rho_3)>0$,
\[
\xi''(\rho_3)=\frac{F'(\rho_3)}{D(\rho_3)^2}>0.
\]
On the other hand, $D(\rho)<0$ for $\rho\in[\rho_5,\beta]$ because
\[
D(0)=\beta(1-2\beta)>0,
\qquad
D(\beta)=-2\beta(1-\beta)(l-1)(1-\beta d)<0,
\]
and
\[
D(\rho_5)=-\frac{\beta l(1-\beta)(1-2\beta)^2(l-1)(1-\beta d)}{(l(1-\beta)-1)^2}<0.
\]
Thus \eqref{eq:xi-second} implies $K(\rho_3)<0$.  Since $K(\rho_3)$ is the maximum of
$K$ on $(\rho_5,\beta)$, $K(\rho)<0$ throughout this interval.  Together with
$D(\rho)<0$, \eqref{eq:xi-second} gives $\xi''(\rho)>0$ on $(\rho_5,\beta)$.
\end{proof}

\begin{lemma}[Shape of the negative-definiteness region]\label{l313}
Assume $d\geq2$, $l\geq3$, and
\begin{equation}
0<\beta\leq \frac1{dl-d-l+2}=\frac1{(d-1)(l-1)+1}.
\label{bl1}
\end{equation}
Then the part of the curve $\det H(\beta,\rho,l\eta)=0$ lying in
$\overline{\widehat{\mathcal R}_\beta}$ is precisely
\begin{equation*}
\Gamma_\beta:=\{(\rho,\xi(\rho)): \rho_5\leq \rho\leq \beta\}.
\end{equation*}
It meets the upper boundary $\eta=b_+(\rho)$ only at
$(\rho_5,\beta-\rho_5)$ and $(\beta,0)$.  Moreover
\begin{equation}
b_-(\rho)\leq \mathcal L(\rho)\leq \xi(\rho)<b_+(\rho),
\qquad \rho_5<\rho<\beta,
\label{eq:line-xi-order}
\end{equation}
where $\mathcal L$ is the line through the saddle and $(\beta,0)$,
\begin{equation}
\mathcal L(\rho):=\frac{(d-1)\beta}{1-\beta}(\beta-\rho).
\label{eq:line-L}
\end{equation}
Consequently $\widehat{\mathcal R}_\beta\cap\{\det H>0\}$ is connected; it is the portion
of $\widehat{\mathcal R}_\beta$ lying below $\Gamma_\beta$, with the convention that for
$0<\rho<\rho_5$ the whole vertical section of $\widehat{\mathcal R}_\beta$ belongs to
$\{\det H>0\}$.
\end{lemma}

\begin{proof}
Let $\rho_1$ be the unique zero of $D$ in $(0,\beta)$.  The uniqueness follows from
$D(0)>0$, $D(\rho_5)<0$, $D(\beta)<0$, and the fact that $D$ is an upward-opening
quadratic because $dl-l-d=(d-1)(l-1)-1>0$; hence $\rho_1<\rho_5$ and $D<0$ on $(\rho_1,\beta)$.  The equation
$\xi(\rho)=b_+(\rho)$ is equivalent to $J(\rho)=D(\rho)$, and
\begin{equation}
J(\rho)-D(\rho)=2(\beta d-1)
\{\beta(1-2\beta)-[l(1-\beta)-1]\rho\}.
\label{eq:J-minus-D}
\end{equation}
Thus $\xi=b_+$ has the unique solution $\rho_5$ in $(0,\beta)$, and also the endpoint
solution $\rho=\beta$.

We next show that no other branch of $\xi$ enters the feasible region.  If $\rho<\rho_1$,
then $D(\rho)>0$ and $\rho<\rho_3$, so \eqref{dxi} gives $\xi'(\rho)<-1$.  If
$\beta\leq1/(2d)$, then $b_-(\rho)=0$ for all $\rho$ and
$\xi(0)=\beta(2\beta d-1)\leq0$, whence $\xi(\rho)<b_-(\rho)$ on $(0,\rho_1)$.
Assume instead that $\beta>1/(2d)$, and put $r_0=2\beta-1/d$.  On
$[0,\min\{r_0,\rho_1\})$ we have $b_-(\rho)=2\beta-\rho-1/d$ and
\[
\frac{d}{d\rho}\{\xi(\rho)-b_-(\rho)\}=\frac{F(\rho)}{D(\rho)^2}<0,
\qquad
\xi(0)-b_-(0)=\frac{(\beta d-1)(2\beta d-1)}d<0.
\]
Thus $\xi<b_-$ on this interval.  If $r_0<\rho_1$, then $\xi(r_0)<0$ and the inequality
$\xi'<-1$ gives $\xi(\rho)<0=b_-(\rho)$ for $r_0<\rho<\rho_1$.  Hence
$\xi(\rho)<b_-(\rho)$ for every $0<\rho<\rho_1$.  For $\rho_1<\rho<\rho_5$, the
right-hand side of \eqref{eq:J-minus-D} is negative while $D<0$, so $J/D>1$ and hence
$\xi(\rho)>b_+(\rho)$.  Therefore the only part of the curve that lies in the closed
feasible region is the branch $\rho\in[\rho_5,\beta]$.

It remains to prove \eqref{eq:line-xi-order}.  First, $\beta\leq1/(dl-d-l+2)\leq1/(d+1)$
for $d\geq2,l\geq3$, and a direct calculation gives
\[
\mathcal L(0)-\left(2\beta-\frac1d\right)
=-\frac{(\beta d-1)(\beta d+\beta-1)}{d(\beta-1)}\geq0.
\]
Since both $\mathcal L(\rho)$ and $2\beta-\rho-1/d$ are affine in $\rho$, and
$\mathcal L'(\rho)=-{(d-1)\beta}/{(1-\beta)}>-1$, we have
$b_-(\rho)\leq\mathcal L(\rho)$ for all $0\leq\rho\leq\beta$.

Next,
\[
\xi'(\beta)=-\frac1{l-1}.
\]
By \eqref{bl1},
\[
-\frac1{l-1}\leq -\frac{(d-1)\beta}{1-\beta}=\mathcal L'(\rho).
\]
Since $\xi'$ is strictly increasing on $(\rho_5,\beta)$ by Lemma \ref{l39},
$\xi'(\rho)<\mathcal L'(\rho)$ for $\rho_5<\rho<\beta$.  As $\xi(\beta)=\mathcal L(\beta)=0$,
integrating from $\rho$ to $\beta$ yields $\mathcal L(\rho)<\xi(\rho)$ for
$\rho_5<\rho<\beta$.  Finally, \eqref{eq:J-minus-D} and $D<0$ imply
$\xi(\rho)<b_+(\rho)$ for $\rho_5<\rho<\beta$.  This proves \eqref{eq:line-xi-order}.

The sign assertion follows because the saddle
$(\beta^2,(d-1)\beta^2)$ lies on $\mathcal L$ and hence below $\Gamma_\beta$; at the
saddle, \eqref{dth} is positive under \eqref{bl1}.  Since the zero set in the feasible
region is exactly the graph $\Gamma_\beta$, the region below this graph is the connected
component on which $\det H>0$.
\end{proof}

\begin{lemma}[Star-shapedness and the maximum on $\det H\geq0$]\label{lm39}
Assume $d\geq2$, $l\geq3$, and \eqref{bl1}.  Then
\begin{equation*}
\max_{(\rho,\theta)\in\mathcal R_\beta\cap\{\det H\geq0\}}
\Psi_{d,l}(\beta,\rho,\theta)=2\Phi_{d,l}(\beta),
\end{equation*}
and the maximizer is unique, namely $(\beta^2,l(d-1)\beta^2)$.
\end{lemma}

\begin{proof}
Work in the $(\rho,\eta)$-plane.  Let
$z_0=(\beta^2,(d-1)\beta^2)$, and let $S_\beta$ be the closed subset of
$\overline{\widehat{\mathcal R}_\beta}$ on which $\det H\geq0$ and which contains
$z_0$.  By Lemma \ref{l313}, $S_\beta$ is the feasible region below the graph
$\Gamma_\beta$.

We first prove that $S_\beta$ is star-shaped with respect to $z_0$.  The feasible region
$\overline{\widehat{\mathcal R}_\beta}$ is convex, so only the determinant boundary
$\Gamma_\beta$ has to be checked.  Fix $u\in[\rho_5,\beta]$.  Because
$\beta\leq1/(dl-d-l+2)\leq1/l$, we have $\beta^2\leq\rho_5\leq u$.  By Lemma \ref{l39},
$\xi$ is convex on $(\rho_5,\beta)$; hence its tangent at $u$ lies below its graph.
Furthermore, the point $z_0$ lies below that tangent.  Indeed, since
$\xi'(u)\leq\mathcal L'$ and $\beta^2-u\leq0$, while
$\mathcal L(u)\leq\xi(u)$,
\[
(d-1)\beta^2=\mathcal L(\beta^2)
=\mathcal L(u)+\mathcal L'(\beta^2-u)
\leq \xi(u)+\xi'(u)(\beta^2-u).
\]
Therefore the segment from $z_0$ to $(u,\xi(u))$ lies below the tangent at $u$, and hence
below the graph of $\xi$.  This proves the star-shapedness.

For any $z\in S_\beta\setminus\{z_0\}$, let $\gamma$ be the line segment from $z_0$ to
$z$.  The interior of $\gamma$ lies in the region where $\det H>0$.  Since
$\partial_{\rho\rho}^2\widehat\Psi_\beta<0$ and $\det H>0$ there, the Hessian of
$\widehat\Psi_\beta$ is negative definite along the interior of $\gamma$.  The first
derivative of $\widehat\Psi_\beta$ in the direction of $\gamma$ is zero at $z_0$ and then
strictly decreases along the segment.  Hence
\[
\widehat\Psi_\beta(z)<\widehat\Psi_\beta(z_0)=2\Phi_{d,l}(\beta)
\]
for every $z\ne z_0$ in $S_\beta$.
\end{proof}

\begin{lemma}[Global maximum for $l\geq3$]\label{lm313}
Assume $d\geq2$, $l\geq3$, \eqref{bl1}, and
\begin{equation}
\Phi_{d,l}(\beta)>0.
\label{eq:Phi-positive}
\end{equation}
Then
\begin{equation*}
\sup_{(\rho,\theta)\in\mathcal R_\beta}
\Psi_{d,l}(\beta,\rho,\theta)=2\Phi_{d,l}(\beta),
\end{equation*}
and $(\beta^2,l(d-1)\beta^2)$ is the unique maximizer in $\mathcal R_\beta$.
\end{lemma}

\begin{proof}
By Lemma \ref{lm39}, the supremum over $\mathcal R_\beta\cap\{\det H\geq0\}$ is
$2\Phi_{d,l}(\beta)$ and is attained only at the saddle.  It remains to consider
$\det H<0$.  An interior maximum in this region is impossible, because at such a point
the Hessian would have to be negative semidefinite, whereas $\det H<0$ makes it
indefinite.  The boundary $\det H=0$ is already controlled by Lemma \ref{lm39}.  By
Lemma \ref{l313}, the remaining boundary of the region $\det H<0$ inside
$\widehat{\mathcal R}_\beta$ is the upper boundary $\eta=\beta-\rho$ with
$\rho\in[\rho_5,\beta]$.

Along this boundary put
\begin{equation}
T_{d,l,\beta}(\rho):=\Psi_{d,l}(\beta,\rho,l(\beta-\rho)).
\label{dtr}
\end{equation}
Then
\begin{align*}
T_{d,l,\beta}(\rho)
&=-\rho\log\rho+(l-1)(1-2\beta+\rho)\log(1-2\beta+\rho)
+l(\beta-\rho)\log(d-1)\\
&\quad +(l-2)(\beta-\rho)\log(\beta-\rho)
-\frac ld(1-\beta d)\log(1-\beta d),
\end{align*}
and
\begin{equation}
T_{d,l,\beta}''(\rho)
=-\frac1\rho+\frac{l-1}{1-2\beta+\rho}+\frac{l-2}{\beta-\rho}
=\frac{[l(1-\beta)-1]\rho-\beta(1-2\beta)}
{\rho(\beta-\rho)(1-2\beta+\rho)}.
\label{dtr2}
\end{equation}
Thus $T_{d,l,\beta}$ is convex on $[\rho_5,\beta)$.  Its supremum on this interval is
therefore the larger of its endpoint limits.  The endpoint at $\rho_5$ lies on
$\det H=0$ and is controlled by Lemma \ref{lm39}.  At $\rho=\beta$,
\[
\lim_{\rho\uparrow\beta}T_{d,l,\beta}(\rho)=\Phi_{d,l}(\beta)<2\Phi_{d,l}(\beta)
\]
by \eqref{eq:Phi-positive}.  Hence the supremum over $\det H<0$ is at most
$2\Phi_{d,l}(\beta)$, and equality can occur only at the saddle.
\end{proof}

\begin{lemma}[Second-moment asymptotic]\label{l08}
Assume
\begin{equation}
\Psi_{d,l}(\beta,\beta^2,l(d-1)\beta^2)=2\Phi_{d,l}(\beta)>0.
\label{bm}
\end{equation}
Assume also that one of the following conditions holds:
\begin{enumerate}
\item $l=2$ and $0<\beta<1/d$; or
\item $l\geq3$ and $0<\beta\leq (dl-d-l+2)^{-1}$.
\end{enumerate}
Then
\begin{equation}
\lim_{m\to\infty}
\frac{\mathbb E Z_{m\beta}^2}{(\mathbb E Z_{m\beta})^2}
=\frac{1-\beta}{\sqrt{\beta^2(d+l-dl)-2\beta+1}}.
\label{cm}
\end{equation}
\end{lemma}

\begin{proof}
Let $z_\beta=(\beta^2,l(d-1)\beta^2)$ and write
$\Psi_\beta(\rho,\theta)=\Psi_{d,l}(\beta,\rho,\theta)$.  Under either hypothesis,
$z_\beta$ is the unique point in $\mathcal R_\beta$ at which the supremum of
$\Psi_\beta$ is equal to $2\Phi_{d,l}(\beta)$, and the Hessian
$H_\beta:=H(\beta,z_\beta)$ is negative definite.  In the case $l=2$ this follows from
Lemma \ref{l04}; in the case $l\ge3$ it follows from Lemma \ref{lm313}.  Moreover
\eqref{eq:Psi-at-saddle} gives $\Psi_\beta(z_\beta)=2\Phi_{d,l}(\beta)$.

We first make the contribution away from the saddle explicit.  Extend
$\Psi_\beta$ continuously to the compact polygon
\[
  \overline{\mathcal R_\beta}=
  \left\{0\le\rho\le\beta,
  \max\left\{0,2l\beta-l\rho-\frac ld\right\}\le \theta\le l\beta-l\rho
  \right\}
\]
by the convention $u\log u=0$ at $u=0$.  The boundary analysis in
Lemmas~\ref{l04} and~\ref{lm313} shows that $z_\beta$ is the unique maximizer of
this continuous extension with value $2\Phi_{d,l}(\beta)$.  Therefore, for every
fixed $\delta>0$ small enough, there exists $\eta_\delta>0$ such that
\begin{equation}\label{eq:l08-away-gap}
  \sup\{\Psi_\beta(z):z\in\overline{\mathcal R_\beta},
        \|z-z_\beta\|\ge\delta\}
  \le 2\Phi_{d,l}(\beta)-\eta_\delta .
\end{equation}
The crude form of Stirling's formula applied to the exact expression \eqref{dfb}
gives a polynomially prefactored upper bound, uniformly over all admissible lattice
points,
\begin{equation}\label{eq:l08-crude-bound}
  F_\beta(s,t)
  \le m^C\exp\left\{m\Psi_\beta(s/m,t/m)+o(m)\right\},
\end{equation}
with $C$ depending only on $d,l$ and $\beta$.  Combining
\eqref{eq:l08-away-gap} and \eqref{eq:l08-crude-bound}, and using that the number of
admissible pairs $(s,t)$ is $O(m^2)$, shows that the contribution from
$\| (s/m,t/m)-z_\beta\|\ge\delta$ is
$O(m^{C+2}e^{m(2\Phi_{d,l}(\beta)-\eta_\delta+o(1))})$, hence is exponentially
smaller than the saddle contribution.

It remains to evaluate a small neighborhood of $z_\beta$.  Choose $\delta>0$ so that
the closed ball $B_\delta(z_\beta)$ is contained in $\mathcal R_\beta$ and the Hessian
of $\Psi_\beta$ is negative definite throughout this ball.  Taylor's formula and
negative definiteness give, uniformly for $\|z-z_\beta\|\leq\delta$,
\[
\Psi_\beta(z)=2\Phi_{d,l}(\beta)+\frac12(z-z_\beta)H_\beta(z-z_\beta)^T
+O(\|z-z_\beta\|^3),
\]
and also
\[
\Psi_\beta(z)\leq 2\Phi_{d,l}(\beta)-c\|z-z_\beta\|^2
\]
for some $c>0$.  Hence the part of this neighborhood outside an
$L/\sqrt m$-neighborhood of $z_\beta$ is bounded by a Gaussian tail, uniformly in $m$,
and vanishes as $L\to\infty$.

On the $L/\sqrt m$-neighborhood, set
\[
x_1=\sqrt m\left(\frac{s}{m}-\beta^2\right),
\qquad
x_2=\sqrt m\left(\frac{t}{m}-l(d-1)\beta^2\right).
\]
Using Lemma \ref{lem:second-stirling} and then the Riemann-sum limit gives
\begin{align*}
\mathbb E Z_{m\beta}^2
&=\frac{e^{2m\Phi_{d,l}(\beta)}}{m}
\frac1{4\pi^2\beta^3(1-d\beta)^2\sqrt{l(d-1)}}
\int_{\mathbb R^2}\exp\left\{\frac12 xH_\beta x^T\right\}\,dx\,(1+o(1))\\
&=\frac{e^{2m\Phi_{d,l}(\beta)}}{m}
\frac{1}{2\pi\beta^3(1-d\beta)^2\sqrt{l(d-1)}\sqrt{\det H_\beta}}
(1+o(1)).
\end{align*}
By Lemma \ref{lem:first-moment-local},
\[
(\mathbb E Z_{m\beta})^2=
\frac{e^{2m\Phi_{d,l}(\beta)}}{m}
\frac1{2\pi\beta(1-d\beta)}(1+o(1)).
\]
Finally, \eqref{dth} gives
\[
\det H_\beta=
\frac{\beta^2(d+l-dl)-2\beta+1}
{l\beta^4(1-d\beta)^2(1-\beta)^2(d-1)}.
\]
Substitution yields \eqref{cm}.
\end{proof}

\begin{corollary}[Compact-uniform second-moment ratio]\label{cor:uniform-second-ratio-L1}
Let $K\subset(0,1/d)$ be compact.  Suppose that either $l=2$, or $l\ge3$ and
$K\subset(0,(dl-d-l+2)^{-1}]$ with $\Phi_{d,l}(\beta)>0$ for every $\beta\in K$.
Then the convergence in \eqref{cm} is uniform for $\beta\in K$ along admissible
subsequences $m\beta\in\mathbb N$.  Equivalently, if $h_m/m\in K$ and
$h_m/m\to\beta\in K$, then
\[
  \frac{\mathbb E Z_{h_m}^2}{(\mathbb E Z_{h_m})^2}
  \longrightarrow
  \frac{1-\beta}{\sqrt{\beta^2(d+l-dl)-2\beta+1}}.
\]
\end{corollary}

\begin{proof}
All ingredients in the proof of Lemma~\ref{l08} are compact-uniform in $\beta$.
The saddle remains a positive distance from the boundary on compact subsets, the
Hessian determinant at the saddle is bounded away from zero by Lemma~\ref{l24},
and the gap estimate \eqref{eq:l08-away-gap} is uniform by compactness and the
uniqueness of the maximizer.  The Stirling estimates in Lemmas~\ref{l12}, \ref{lem:first-moment-local}, and
\ref{lem:second-stirling} are uniform on compact sets.  The same discrete Laplace
argument therefore gives the stated uniform convergence.
\end{proof}

\section{Subgraph Conditioning}\label{sc}

Throughout this section \(d,l\ge2\) are fixed, \(lm/d\in\mathbb N\), and
\(m\to\infty\) through admissible values.  We also fix
\(\beta\in(0,1/d)\) and consider only subsequences for which \(m\beta\in\mathbb N\).
The estimates below are uniform for fixed choices of all cycle lengths and
multiplicities.  More precisely, the error terms are compact-uniform in
\(\beta\) as long as \(\beta\) stays in a compact subset of \((0,1/d)\); this
uniformity is recorded explicitly at the end of the section because it is used
in Section~\ref{fe} for density sequences \(h_m/m\to\beta\).

It is convenient to view a \((d,l)\)-regular hypergraph as its incidence
multigraph: one class consists of the \(m\) hyperedges, the other class consists
of the \(lm/d\) vertices, and each half-edge gives one incidence between a
hyperedge and a vertex.  A \(k\)-cycle of the hypergraph is a cycle of length
\(2k\) in this incidence multigraph.  Equivalently, it is an alternating sequence
\[
  v_0,e_1,v_1,e_2,\ldots,v_{k-1},e_k,v_k=v_0,
\]
where the \(e_i\) are distinct hyperedges, the \(v_0,\ldots,v_{k-1}\) are distinct
vertices, and the two incidences of \(e_i\) with \(v_{i-1}\) and \(v_i\) are realised
by two distinct half-edges of \(e_i\).  For \(k=1\) this means a pair of distinct
half-edges of one hyperedge which are incident to the same vertex.  Thus a
hyperedge with \(r\) half-edges at one vertex contributes \(\binom r2\) one-cycles.
Let \(C_k\) be the number of \(k\)-cycles.

For \(b\ge1\) and \(\mathbf c=(c_1,\ldots,c_b)\in\mathbb N^b\), put
\[
  \Omega_{\mathbf c}
  :=\{G\in\Omega_{m,d,l}: C_k(G)=c_k,\ 1\le k\le b\},
\]
and
\[
  E_{\beta,\mathbf c}
  :=\mathbb E\bigl(Z_{m\beta}\mid G_{m,d,l}\in\Omega_{\mathbf c}\bigr).
\]
Let \(M_0\subseteq[m]\) be a fixed set of \(m\beta\) hyperedges, and let
\(A_{M_0}\) be the event that \(M_0\) is a matching.  By symmetry of the labelled
hyperedges, for every \(M\subseteq[m]\) with \(|M|=m\beta\),
\[
  \mathbb P(M\text{ is a matching})=\mathbb P(A_{M_0}),
\]
and, conditional on the event that \(M\) is a matching, the distribution of the
cycle vector \((C_1,\ldots,C_b)\) is the same as under \(A_{M_0}\).  Hence
\begin{equation}\label{eq:ebc-exact}
  E_{\beta,\mathbf c}
  =\mathbb E Z_{m\beta}\,
    \frac{\mathbb P\bigl((C_1,\ldots,C_b)=\mathbf c\mid A_{M_0}\bigr)}
         {\mathbb P\bigl((C_1,\ldots,C_b)=\mathbf c\bigr)}.
\end{equation}
Indeed,
\[
  E_{\beta,\mathbf c}
  =\sum_{|M|=m\beta}\mathbb P(M\text{ is a matching}
     \mid (C_1,\ldots,C_b)=\mathbf c),
\]
and Bayes' rule gives \eqref{eq:ebc-exact}.

We first compute the limiting mean of \(C_k\) under the conditioning \(A_{M_0}\).
Fix \(k\ge1\).  Let \(t\) be the number of hyperedges of the cycle that lie in
\(M_0\).  Since \(M_0\) is a matching, no two such hyperedges are adjacent along
the cycle; hence \(0\le t\le\lfloor k/2\rfloor\).  If \(t\ge1\), let \(s\) be the
number of cycle vertices which are not incident, along the cycle, to a
hyperedge of \(M_0\), but which contain one further half-edge belonging to
\(M_0\).  Then \(0\le s\le k-2t\).  The case \(t=0\) is interpreted in the same way,
with \(0\le s\le k\).

Let \(\mathcal E^{(m)}_{\beta,k}(t,s)\) be the contribution to
\(\mathbb E(C_k\mid A_{M_0})\) from cycles of type \((t,s)\).  A direct exposure of
the finitely many vertices on the cycle gives the following exact expressions.
For \(t\ge1\),
\begin{align}
\mathcal E^{(m)}_{\beta,k}(t,s)
&=\frac{(m\beta)!(k-t-1)!}{2t!(m\beta-t)!}[l(l-1)]^k(d-1)^k(d-2)^s
   \frac{(m-m\beta)!}{(m-m\beta-k+t)!} \notag\\
&\quad\times
   \frac{(lm\beta-2t)!}{(lm\beta-2t-s)!(k-2t-s)!s!}
   \frac{(lm-lm\beta-2k+2t)!}{(lm(1-\beta))!} \notag\\
&\quad\times
   d^{k-2t-s}
   \frac{\bigl(\frac{lm}{d}(1-d\beta)\bigr)!}
        {\bigl(\frac{lm}{d}-(lm\beta+k-2t-s)\bigr)!} .
\label{eq:Ebeta-k-ts}
\end{align}
For \(t=0\),
\begin{align}
\mathcal E^{(m)}_{\beta,k}(0,s)
&=\frac{(m-m\beta)!}{2(m-m\beta-k)!}(k-1)![l(l-1)]^k(d-1)^k(d-2)^s
   \frac{(lm\beta)!}{s!(lm\beta-s)!(k-s)!} \notag\\
&\quad\times
   \frac{(lm-lm\beta-2k)!}{(lm(1-\beta))!}
   d^{k-s}
   \frac{\bigl(\frac{lm}{d}(1-d\beta)\bigr)!}
        {\bigl(\frac{lm}{d}-(lm\beta+k-s)\bigr)!} .
\label{eq:Ebeta-k-0s}
\end{align}
The factors have the following meaning.  We choose and cyclically arrange the
cycle hyperedges, choose the ordered pair of half-edges used by the cycle in
each chosen hyperedge, choose the vertices at which adjacent cycle half-edges
meet, and finally divide by the number of completions compatible with
\(A_{M_0}\).  Since only \(O(1)\) half-edges are exposed, the factorial ratios above
are valid for all sufficiently large \(m\).
Consequently
\begin{equation}\label{eck}
  \mathbb E(C_k\mid A_{M_0})
  =\sum_{t=0}^{\lfloor k/2\rfloor}\sum_{s=0}^{k-2t}
     \mathcal E^{(m)}_{\beta,k}(t,s).
\end{equation}

\begin{lemma}\label{l30}
For fixed \(k\),
\begin{equation}\label{eq:cond-cycle-mean}
  \mathbb E(C_k\mid A_{M_0})
  \longrightarrow
  \mu_{\beta,k}
  :=\frac{(d-1)^k(l-1)^k}{2k}
    \left(1+\frac{(-1)^k\beta^k}{(1-\beta)^k}\right).
\end{equation}
\end{lemma}

\begin{proof}
From \eqref{eq:Ebeta-k-ts}, for \(t\ge1\) and fixed \(t,s,k\),
\begin{align*}
\mathcal E^{(m)}_{\beta,k}(t,s)
&=(1+o(1))\,
\frac{(d-1)^k(l-1)^k(1-d\beta)^k}{2(1-\beta)^k}
\binom{k-t-1}{t-1}\binom{k-2t}{s}\frac1t \\
&\quad\times
\frac{\beta^{t+s}(1-\beta)^t(d-2)^s}{(1-d\beta)^{2t+s}}.
\end{align*}
Equivalently, since
\(t^{-1}\binom{k-t-1}{t-1}=(k-t)^{-1}\binom{k-t}{t}\),
\begin{align*}
\sum_{s=0}^{k-2t}\mathcal E^{(m)}_{\beta,k}(t,s)
&=(1+o(1))\,
\frac{(d-1)^k(l-1)^k}{2(1-\beta)^k}
\frac1{k-t}\binom{k-t}{t} \\
&\quad\times
\beta^t(1-\beta)^t(1-2\beta)^{k-2t}.
\end{align*}
Similarly, \eqref{eq:Ebeta-k-0s} gives
\[
  \sum_{s=0}^{k}\mathcal E^{(m)}_{\beta,k}(0,s)
  =(1+o(1))\,
  \frac{(d-1)^k(l-1)^k}{2k(1-\beta)^k}(1-2\beta)^k .
\]
Thus, with
\[
  \alpha=\frac{\beta(1-\beta)}{(1-2\beta)^2},
\]
we obtain
\begin{align*}
\mathbb E(C_k\mid A_{M_0})
&=(1+o(1))\,
\frac{(d-1)^k(l-1)^k(1-2\beta)^k}{2(1-\beta)^k}
\sum_{t=0}^{\lfloor k/2\rfloor}
  \frac1{k-t}\binom{k-t}{t}\alpha^t .
\end{align*}
The elementary coefficient identity
\[
\sum_{t=0}^{\lfloor k/2\rfloor}\frac1{k-t}\binom{k-t}{t}\alpha^t
 =\frac{(1-\beta)^k+(-1)^k\beta^k}{k(1-2\beta)^k}
\]
follows, for example, by extracting the coefficient of \(x^k\) in
\(-\log(1-x(1+x)/\alpha)\).  Substituting this identity proves
\eqref{eq:cond-cycle-mean}.
\end{proof}

Let
\begin{equation}\label{elk}
  \lambda_k:=\frac{(d-1)^k(l-1)^k}{2k},\qquad k\ge1.
\end{equation}
Then \eqref{eq:cond-cycle-mean} can be written as
\begin{equation}\label{ecc}
  \mu_{\beta,k}=\lambda_k\left(1+\frac{(-1)^k\beta^k}{(1-\beta)^k}\right).
\end{equation}

Let \(N_{M_0}\) be the number of configurations in which \(M_0\) is a matching.  By
\eqref{ezb}, or directly by completing the vertices containing the half-edges
of \(M_0\) and then partitioning the remaining half-edges,
\[
  N_{M_0}
  =\frac{(lm(1-\beta))!}
    {d^{\frac{lm}{d}(1-d\beta)}\bigl(\frac{lm}{d}(1-d\beta)\bigr)!}.
\]

For fixed distinct positive integers \(k_1,\ldots,k_g\) and non-negative integers
\(r_1,\ldots,r_g\), let \(N^{\mathrm{dis}}_{(r_1,k_1),\ldots,(r_g,k_g)}\) be the
number of configurations counted by \(N_{M_0}\) together with, for each \(i\), an
ordered list of \(r_i\) labelled \(k_i\)-cycles, such that all these selected cycles
are pairwise disjoint in the incidence multigraph.  Let
\(N^{\mathrm{ov}}_{(r_1,k_1),\ldots,(r_g,k_g)}\) be the corresponding number of
configurations with at least one intersection among the selected cycles.
Here an intersection means sharing at least one incidence-graph vertex, i.e.
sharing a hyperedge or a hypergraph vertex.

\begin{lemma}\label{l31}
For fixed distinct \(k_1,\ldots,k_g\) and fixed \(r_1,\ldots,r_g,d,l\),
\[
  \frac{N^{\mathrm{dis}}_{(r_1,k_1),\ldots,(r_g,k_g)}}{N_{M_0}}
  \longrightarrow
  \prod_{i=1}^g \mu_{\beta,k_i}^{\,r_i}.
\]
\end{lemma}

\begin{proof}
Expose the selected cycles one at a time, recording for each cycle the pair
\((t,s)\) used above.  Since the total number of exposed hyperedges and vertices
is bounded independently of \(m\), after any finite number of disjoint cycles has
been exposed all populations appearing in \eqref{eq:Ebeta-k-ts} and
\eqref{eq:Ebeta-k-0s} are changed by only \(O(1)\).  Therefore the conditional
contribution of the next disjoint \(k\)-cycle of type \((t,s)\) is
\(\mathcal E^{(m)}_{\beta,k}(t,s)(1+O(m^{-1}))\), uniformly over the finitely many
possible types.  Multiplying these factors and summing over the finitely many
type choices gives
\[
  \frac{N^{\mathrm{dis}}_{(r_1,k_1),\ldots,(r_g,k_g)}}{N_{M_0}}
  =\prod_{i=1}^g
   \left(\sum_{t=0}^{\lfloor k_i/2\rfloor}\sum_{s=0}^{k_i-2t}
      \mathcal E^{(m)}_{\beta,k_i}(t,s)\right)^{r_i}(1+o(1)).
\]
The result follows from Lemma \ref{l30}.
\end{proof}

\begin{lemma}\label{l32}
For fixed distinct \(k_1,\ldots,k_g\) and fixed \(r_1,\ldots,r_g,d,l\),
\[
  \frac{N^{\mathrm{ov}}_{(r_1,k_1),\ldots,(r_g,k_g)}}{N_{M_0}}\longrightarrow0.
\]
\end{lemma}

\begin{proof}
The total length of the selected cycles is fixed.  If the selected cycles are
not disjoint in the incidence multigraph, then their union has at least one
identification among the hyperedge- or vertex-nodes which would be distinct in
the disjoint case.  Thus the number of choices for the labelled hyperedges and
vertices of the union is smaller by a factor \(O(m^{-1})\) than in the disjoint
case.  The probability that the required finite incidence pattern is realised,
conditioned on \(A_{M_0}\), is of the same order as for a disjoint pattern with
the same number of incidences, because only \(O(1)\) half-edges are exposed and
all remaining pools have size \(\Theta(m)\).  The constants in these
\(O(\cdot)\) bounds depend only on the fixed cycle lengths and multiplicities;
if \(\beta\) is restricted to a compact subset of \((0,1/d)\), all exposed
matching and non-matching half-edge pools have size \(\Theta(m)\) uniformly in
\(\beta\).  Thus each finite overlap pattern has probability at most a constant
multiple of the corresponding disjoint-pattern probability after accounting for
the lost free vertex or hyperedge choice.  Summing over the finitely many
overlap patterns gives
\(N^{\mathrm{ov}}_{(r_1,k_1),\ldots,(r_g,k_g)}=O(m^{-1})
N^{\mathrm{dis}}_{(r_1,k_1),\ldots,(r_g,k_g)}\), which proves the lemma.
\end{proof}

\begin{lemma}\label{l33}
For each fixed \(b\in\mathbb N\) the following hold as \(m\to\infty\).
\begin{enumerate}
\item The vector \((C_1,\ldots,C_b)\) converges in distribution to a vector of
independent Poisson random variables with means \(\lambda_1,\ldots,\lambda_b\).
\item Conditional on \(A_{M_0}\), the vector \((C_1,\ldots,C_b)\) converges in
distribution to a vector of independent Poisson random variables with means
\(\mu_{\beta,1},\ldots,\mu_{\beta,b}\).
\end{enumerate}
\end{lemma}

\begin{proof}
We prove the conditional statement; the unconditional one is the same argument
with no distinguished matching, equivalently with \(\beta=0\), in which case
\(\mu_{0,k}=\lambda_k\).  For arbitrary fixed distinct \(k_1,\ldots,k_g\le b\) and
\(r_1,\ldots,r_g\ge0\),
\[
  \mathbb E\left[\prod_{i=1}^g (C_{k_i})_{r_i}\,\middle|\,A_{M_0}\right]
  =\frac{N^{\mathrm{dis}}_{(r_1,k_1),\ldots,(r_g,k_g)}
        +N^{\mathrm{ov}}_{(r_1,k_1),\ldots,(r_g,k_g)}}{N_{M_0}},
\]
where \((x)_r=x(x-1)\cdots(x-r+1)\).  Lemmas \ref{l31} and \ref{l32} show that
these mixed factorial moments converge to
\(\prod_i\mu_{\beta,k_i}^{r_i}\), which are the mixed factorial moments of
independent Poisson variables.  Since Poisson laws are determined by their
factorial moments, the conditional convergence follows.
\end{proof}

Combining \eqref{eq:ebc-exact} with Lemma \ref{l33}, we obtain the following
fixed-\(b\) asymptotic.  For each fixed
\(\mathbf c=(c_1,\ldots,c_b)\in\mathbb N^b\),
\begin{equation}\label{ebc}
  E_{\beta,\mathbf c}
  =\left(\prod_{k=1}^b
      e^{\lambda_k-\mu_{\beta,k}}
      \left(\frac{\mu_{\beta,k}}{\lambda_k}\right)^{c_k}\right)
    \mathbb E Z_{m\beta}\,(1+o(1)).
\end{equation}

We shall also need three elementary consequences of these Poisson limits.

\begin{lemma}\label{lmp}
Let \(X\) be a Poisson random variable with mean \(\mu\).  Then, for
\(0<\varepsilon<1\),
\[
  \mathbb P\{X<\mu(1-\varepsilon)\}\le e^{-\mu\varepsilon^2/2},
\]
and, for \(\varepsilon>0\),
\[
  \mathbb P\{X>\mu(1+\varepsilon)\}
  \le \left[e^\varepsilon(1+\varepsilon)^{-(1+\varepsilon)}\right]^\mu .
\]
\end{lemma}

\begin{proof}
This is the standard Chernoff bound for Poisson random variables; see, for
example, Theorem A.15 of \cite{as92}.
\end{proof}

\begin{lemma}\label{lem:poisson-upper-tail-uniform}
Let \(a>0\) and \(W>2\).  For every \(y>0\) and every \(\lambda>0\),
\begin{equation}\label{eiq}
\left(\frac{e^{\frac{y}{a\sqrt\lambda}}}
{\left(1+\frac{y}{a\sqrt\lambda}\right)^{1+\frac{y}{a\sqrt\lambda}}}
\right)^\lambda
\le
 e^{-\frac{y^2}{4a^2}}
 +e^{-\frac{|\phi(1/2)|y^2}{a^2W^2}}
 +e^{-(\log(1+W)-1)\frac{y\sqrt\lambda}{a}},
\end{equation}
where
\[
  \phi(z):=z-(1+z)\log(1+z).
\]
\end{lemma}

\begin{proof}
Put \(z=y/(a\sqrt\lambda)\).  The left side is \(e^{\lambda\phi(z)}\).
If \(0<z\le1/2\), then
\(\log(1+z)\ge z-z^2/2\), and hence
\(\phi(z)\le -z^2/2+z^3/2\le -z^2/4\).  This gives the first term in
\eqref{eiq}.  If \(1/2<z\le W\), then \(\phi\) is decreasing on \((0,\infty)\), so
\(\phi(z)\le\phi(1/2)<0\); since \(\lambda=y^2/(a^2z^2)\ge y^2/(a^2W^2)\), this
gives the second term.  If \(z>W\), then
\(\phi(z)\le -z(\log(1+z)-1)\le -z(\log(1+W)-1)\), giving the third term.
\end{proof}

For the next two lemmas set
\begin{equation}\label{eq:q-r-def}
  q_\beta:=\frac{\beta}{1-\beta},\qquad
  r_\beta:=(d-1)(l-1)q_\beta^2,
\end{equation}
and note that \(q_\beta<1\) because \(\beta<1/d\le1/2\).  For fixed \(b\) and \(y>0\),
define
\begin{equation}\label{dsy}
  S_b(y):=\left\{\mathbf c\in\mathbb N^b:
  |c_k-\lambda_k|\le y\sqrt{\lambda_k},\ 1\le k\le b\right\}.
\end{equation}
When no confusion is possible we write \(S(y)\) for \(S_b(y)\).

\begin{lemma}\label{l35}
Assume \(r_\beta<1\).  Fix \(b\in\mathbb N\) and \(W>2\).  Put
\[
  \gamma_\beta:=1-q_\beta,
  \qquad A_\beta:=1+q_\beta,
  \qquad \lambda_{*,b}:=\min_{1\le k\le b}\lambda_k,
\]
and
\begin{align*}
\Theta_{\beta,b,W}(y)
:=1&-b\exp\left\{-\frac{\gamma_\beta^2y^2}{8A_\beta^4}\right\}
    -b\exp\left\{-\frac{\gamma_\beta^2y^2}{16A_\beta^4}\right\} \\
&-b\exp\left\{-\frac{|\phi(1/2)|\gamma_\beta^2y^2}{4A_\beta^4W^2}\right\}
    -b\exp\left\{-(\log(1+W)-1)
       \frac{\gamma_\beta^2y\sqrt{\lambda_{*,b}}}{2A_\beta^2}\right\}.
\end{align*}
Then, for all sufficiently large \(y\),
\begin{align*}
\liminf_{m\to\infty}
\frac{\sum_{\mathbf c\in S_b(y)}p_{\mathbf c}E_{\beta,\mathbf c}^2}
     {(\mathbb E Z_{m\beta})^2}
\ge
\Theta_{\beta,b,W}(y)
\left(1-r_\beta^b\right)
\frac1{\sqrt{1-r_\beta}},
\end{align*}
where \(p_{\mathbf c}=\mathbb P((C_1,\ldots,C_b)=\mathbf c)\).
Equivalently,
\[
  \frac1{\sqrt{1-r_\beta}}
  =\frac{1-\beta}{\sqrt{1-2\beta-(dl-d-l)\beta^2}}.
\]
\end{lemma}

\begin{proof}
By \eqref{ebc} and Lemma \ref{l33}, for fixed \(b\) and \(y\),
\begin{align*}
\lim_{m\to\infty}
\frac{\sum_{\mathbf c\in S_b(y)}p_{\mathbf c}E_{\beta,\mathbf c}^2}
     {(\mathbb E Z_{m\beta})^2}
&=\prod_{k=1}^b
  e^{(\mu_{\beta,k}-\lambda_k)^2/\lambda_k}
  \mathbb P\left\{|X_k-\lambda_k|\le y\sqrt{\lambda_k}\right\},
\end{align*}
where the \(X_k\) are independent Poisson variables with means
\[
  \nu_k:=\frac{\mu_{\beta,k}^2}{\lambda_k}
  =\lambda_k\left(1+(-1)^kq_\beta^k\right)^2 .
\]
Let \(a_k:=1+(-1)^kq_\beta^k\).  Then
\(\gamma_\beta\le a_k\le A_\beta\).  For sufficiently large \(y\), the lower-tail
parameter comparing \(\lambda_k-y\sqrt{\lambda_k}\) with \(\nu_k\) is at least
\(y/(2A_\beta^2\sqrt{\lambda_k})\) whenever the lower endpoint is non-negative;
if the endpoint is negative, the lower tail is empty.  Lemma \ref{lmp} gives
\[
  \mathbb P\{X_k<\lambda_k-y\sqrt{\lambda_k}\}
  \le \exp\left\{-\frac{\gamma_\beta^2y^2}{8A_\beta^4}\right\}.
\]
For the upper tail, the corresponding upper-tail parameter is at least
\(y/(2A_\beta^2\sqrt{\lambda_k})\).  Applying \eqref{eiq} with
\(\lambda\) replaced by \(\gamma_\beta^2\lambda_k\) and
\(y\) replaced by \(\gamma_\beta y/(2A_\beta^2)\) gives
\begin{align*}
\mathbb P\{X_k>\lambda_k+y\sqrt{\lambda_k}\}
&\le
\exp\left\{-\frac{\gamma_\beta^2y^2}{16A_\beta^4}\right\}
+
\exp\left\{-\frac{|\phi(1/2)|\gamma_\beta^2y^2}{4A_\beta^4W^2}\right\} \\
&\quad+
\exp\left\{-(\log(1+W)-1)
       \frac{\gamma_\beta^2y\sqrt{\lambda_k}}{2A_\beta^2}\right\}.
\end{align*}
By the union bound over \(1\le k\le b\),
\[
  \prod_{k=1}^b
  \mathbb P\left\{|X_k-\lambda_k|\le y\sqrt{\lambda_k}\right\}
  \ge \Theta_{\beta,b,W}(y).
\]
Finally,
\[
  \prod_{k=1}^b e^{(\mu_{\beta,k}-\lambda_k)^2/\lambda_k}
  =\exp\left\{\sum_{k=1}^b\frac{r_\beta^k}{2k}\right\}
  =\frac1{\sqrt{1-r_\beta}}
     \exp\left\{-\sum_{k>b}\frac{r_\beta^k}{2k}\right\}.
\]
Since
\(\sum_{k>b} r_\beta^k/(2k)\le -\log(1-r_\beta^b)\) for \(0<r_\beta<1\), the last
exponential is at least \(1-r_\beta^b\).  This proves the lemma.
\end{proof}

\begin{lemma}\label{l36}
Let \(S_b(y)\) be defined by \eqref{dsy}.  Then, for every fixed \(b\) and \(W>2\),
\begin{align*}
\limsup_{m\to\infty}\mathbb P\bigl((C_1,\ldots,C_b)\notin S_b(y)\bigr)
&\le
b e^{-y^2/2}+b e^{-y^2/4}+b e^{-|\phi(1/2)|y^2/W^2} \\
&\quad+
 b\exp\left\{-(\log(1+W)-1)y\sqrt{\lambda_{*,b}}\right\},
\end{align*}
where \(\lambda_{*,b}=\min_{1\le k\le b}\lambda_k\).
In particular, for fixed \(b\), the right-hand side tends to \(0\) as \(y\to\infty\).
\end{lemma}

\begin{proof}
By Lemma \ref{l33}(1), \((C_1,\ldots,C_b)\) converges to independent Poisson
variables with means \(\lambda_1,\ldots,\lambda_b\).  The lower-tail bound in
Lemma \ref{lmp} gives \(e^{-y^2/2}\) for each \(k\).  The upper-tail bound in
Lemma \ref{lmp}, followed by \eqref{eiq} with \(a=1\), gives the remaining three
terms.  A union bound over \(k=1,\ldots,b\) completes the proof.
\end{proof}

\begin{lemma}\label{l38}
Assume \(r_\beta<1\), equivalently
\[
  \beta<\frac1{1+\sqrt{(d-1)(l-1)}}.
\]
There exist constants \(A,B>0\), depending only on \(d,l,\beta\), such that for
every fixed \(b\), every \(\mathbf c\in S_b(y)\), and all sufficiently large \(m\),
\[
  E_{\beta,\mathbf c}\ge e^{-(A+By)}\,\mathbb E Z_{m\beta}.
\]
\end{lemma}

\begin{proof}
Let \(\delta_k=(-1)^kq_\beta^k\).  By \eqref{ebc}, uniformly over the finite set
\(S_b(y)\),
\[
  \frac{E_{\beta,\mathbf c}}{\mathbb E Z_{m\beta}}
  =(1+o(1))\prod_{k=1}^b e^{-\lambda_k\delta_k}(1+\delta_k)^{c_k}.
\]
Write \(c_k=\lambda_k+x_k\sqrt{\lambda_k}\) with \(|x_k|\le y\).  Then
\begin{align*}
\log\prod_{k=1}^b e^{-\lambda_k\delta_k}(1+\delta_k)^{c_k}
&=\sum_{k=1}^b \lambda_k\{\log(1+\delta_k)-\delta_k\}
  +\sum_{k=1}^b x_k\sqrt{\lambda_k}\log(1+\delta_k).
\end{align*}
Since \(|\delta_k|\le q_\beta^k\) and \(q_\beta<1\),
\[
  \log(1+u)-u\ge -\frac{u^2}{1-q_\beta},\qquad |u|\le q_\beta,
\]
and
\[
  |\log(1+u)|\le \frac{|u|}{1-q_\beta},\qquad |u|\le q_\beta.
\]
Therefore
\[
\sum_{k=1}^b \lambda_k\{\log(1+\delta_k)-\delta_k\}
\ge -\frac1{1-q_\beta}\sum_{k=1}^{\infty}\lambda_k q_\beta^{2k}>-\infty,
\]
and
\[
\sum_{k=1}^b x_k\sqrt{\lambda_k}\log(1+\delta_k)
\ge -\frac{y}{1-q_\beta}
    \sum_{k=1}^{\infty}\sqrt{\lambda_k}\,q_\beta^k> -\infty\cdot y.
\]
The last series converges exactly when \(r_\beta<1\).  Absorbing the uniform
\(o(1)\) into the constant \(A\) proves the result.
\end{proof}

\begin{lemma}[Compact-uniform cycle estimates]\label{lem:cycle-uniformity}
Fix $b\in\mathbb N$ and fixed cycle lengths and multiplicities.  Let
$K\subset(0,1/d)$ be compact.  Then the limits in Lemmas~\ref{l30}--\ref{l33}
and the fixed-$b$ asymptotic \eqref{ebc} hold uniformly for
$\beta\in K$ along admissible subsequences $m\beta\in\mathbb N$.  If, in addition,
$r_\beta$ is bounded away from $1$ on $K$, then the estimates in Lemmas~\ref{l35}
and~\ref{l38} hold uniformly for $\beta\in K$ after allowing the constants in
Lemma~\ref{l38} to depend on $K$.
\end{lemma}

\begin{proof}
The proofs expose only finitely many hyperedges, vertices and half-edges.  For
$\beta\in K$, every pool appearing in the exposure process has cardinality
$c m+O(1)$ with $c$ bounded above and below away from zero uniformly in
$\beta$.  Hence every factorial ratio used in
\eqref{eq:Ebeta-k-ts}--\eqref{eq:Ebeta-k-0s} admits a uniform expansion of the
form $(cm)^r(1+O_K(m^{-1}))$.  This gives the uniform convergence of the
conditional means, the disjoint factorial moments, the overlapping estimates and
therefore the Poisson limits.  Formula \eqref{ebc} is obtained from the same
factorial-moment convergence by Bayes' rule, so it is also uniform for fixed
$b$ and fixed $\mathbf c$.

If $\sup_{\beta\in K}r_\beta<1$, then also
$\sup_{\beta\in K}q_\beta<1$ and the series
$\sum_k\lambda_k q_\beta^{2k}$ and
$\sum_k\sqrt{\lambda_k}q_\beta^k$ are uniformly convergent on $K$.  The Poisson
tail estimates in Lemmas~\ref{l35} and~\ref{l38} are then uniform after replacing
$\gamma_\beta=1-q_\beta$ by its positive infimum on $K$ and
$A_\beta=1+q_\beta$ by its finite supremum on $K$.
\end{proof}

\section{Free Energy}\label{fe}

In this section we pass from the fixed-density moment estimates and the subgraph
conditioning estimates to convergence in probability of the free energy.  We write the
fixed-density statement in a form which avoids integrality issues: if
$h_m\in\mathbb N$ and $\beta_m=h_m/m$, then $Z_{h_m}$ denotes the number of
matchings with exactly $h_m$ hyperedges.

Recall the notation
\begin{equation}\label{eq:section5-L1}
  L_1:=\frac1{dl-d-l+2}=\frac1{(d-1)(l-1)+1}.
\end{equation}
For $\beta\in(0,1/d)$ set, as in \eqref{eq:q-r-def},
\begin{equation}\label{eq:section5-q-r}
  q_\beta:=\frac{\beta}{1-\beta},\qquad
  r_\beta:=(d-1)(l-1)q_\beta^2.
\end{equation}
The condition $r_\beta<1$ is equivalent to
$\beta<\bigl(1+\sqrt{(d-1)(l-1)}\bigr)^{-1}$.

\begin{lemma}\label{l51}
Let $h_m\in\mathbb N$, put $\beta_m=h_m/m$, and assume that
$\beta_m\to\beta\in(0,1/d)$.  Suppose that one of the following alternatives holds:
\begin{enumerate}
\item $l=2$; or
\item $l\ge3$ and $\beta_m\le L_1$ for all sufficiently large $m$.
\end{enumerate}
Then
\begin{equation}\label{cfe}
  \frac1m\log Z_{h_m}\longrightarrow \Phi_{d,l}(\beta)
\end{equation}
in probability.  In particular, if $m\beta\in\mathbb N$ along a subsequence, the
same conclusion holds with $h_m=m\beta$.
\end{lemma}

\begin{proof}
We first check that $\Phi_{d,l}(\beta)>0$.  If $l=2$, this follows from
Lemma \ref{lm21}(2), since $f_2(1/d)\ge0$.  If $l\ge3$, then
$L_1=1/(dl-d-l+2)$.  Put
\[
B:=dl-d-l+2=(d-1)(l-1)+1.
\]
For $0<u<1$ we have $(1-u)\log(1-u)\ge -u$ and
$-(1-u)\log(1-u)\ge u(1-u)$.  Hence, for every $0<\beta'\le 1/B$,
\begin{align*}
\Phi_{d,l}(\beta')
&\ge \beta'\log\frac1{\beta'}-(l-1)\beta'
   +l\beta'(1-d\beta') \\
&=\beta'\left(\log\frac1{\beta'}+1-ld\beta'\right)
\ge \beta'\left(\log B+1-\frac{ld}{B}\right)>0.
\end{align*}
The last inequality uses $B\ge3$ and $ld/B\le2$, which is equivalent to
$(d-2)(l-2)\ge0$.  Since $\beta_m\to\beta$ and $\beta_m\le L_1$ eventually in
case $l\ge3$, this proves the desired positivity.

By the compact-uniform first moment estimate in Lemma \ref{lem:first-moment-local},
\begin{equation}\label{eq:section5-first-moment-log}
  \frac1m\log \mathbb E Z_{h_m}\longrightarrow \Phi_{d,l}(\beta).
\end{equation}
It remains to prove
\begin{equation}\label{lmb}
  \frac1m\log\frac{Z_{h_m}}{\mathbb E Z_{h_m}}\longrightarrow0
  \quad\text{in probability}.
\end{equation}
The assumptions also imply $r_\beta<1$.  If $l=2$, then for $d=2$ we have
$r_\beta=q_\beta^2<1$, while for $d\ge3$ the inequality $\beta<1/d$ gives
$q_\beta<1/(d-1)$ and hence $r_\beta=(d-1)q_\beta^2<1$.  If $l\ge3$, then
$(d-1)(l-1)>1$ and the bound $\beta_m\le L_1$ implies
$q_\beta\le 1/((d-1)(l-1))$, whence
$r_\beta\le 1/((d-1)(l-1))<1$.

Write
\[
  X_m:=Z_{h_m},\qquad \mathcal F_{m,b}:=\sigma(C_1,\ldots,C_b),\qquad
  X_{m,b}:=\mathbb E(X_m\mid\mathcal F_{m,b}).
\]
All cycle estimates below are applied with $\beta=\beta_m$; the compact-uniform
form needed for this varying-density use is supplied by Lemma~\ref{lem:cycle-uniformity}.

Fix $b\ge1$.  For every fixed $\mathbf c=(c_1,\ldots,c_b)\in\mathbb N^b$, the
asymptotic formula \eqref{ebc}, applied with $\beta_m$ in place of $\beta$, gives
\begin{equation}\label{eq:section5-Ebc-limit}
  \frac{E_{\beta_m,\mathbf c}}{\mathbb E Z_{h_m}}
  \longrightarrow
  \prod_{k=1}^b e^{-\lambda_k\delta_k}(1+\delta_k)^{c_k},
  \qquad
  \delta_k:=(-1)^kq_\beta^k .
\end{equation}
The convergence is uniform over finite sets of $\mathbf c$'s by
Lemma~\ref{lem:cycle-uniformity}.  Lemma
\ref{l33}(1) gives
\[
  (C_1,\ldots,C_b)\Rightarrow (P_1,\ldots,P_b),
\]
where $P_1,\ldots,P_b$ are independent Poisson random variables with means
$\lambda_1,\ldots,\lambda_b$.  Combining this convergence with
\eqref{eq:section5-Ebc-limit} and tightness of $(C_1,\ldots,C_b)$ yields
\begin{equation}\label{eq:section5-Xmb-limit}
  \frac{X_{m,b}}{\mathbb E X_m}\Rightarrow
  Y_b:=\prod_{k=1}^b e^{-\lambda_k\delta_k}(1+\delta_k)^{P_k}.
\end{equation}
Since $q_\beta<1$, all factors in $Y_b$ are positive.

The sequence $Y_b$ converges almost surely to a finite strictly positive random
variable.  Indeed,
\begin{align*}
\log Y_b
&=\sum_{k=1}^b (P_k-\lambda_k)\log(1+\delta_k)
  +\sum_{k=1}^b \lambda_k\{\log(1+\delta_k)-\delta_k\}.
\end{align*}
Because
\begin{equation}\label{eq:section5-l2-summable}
  \sum_{k\ge1}\lambda_k\delta_k^2
  =\sum_{k\ge1}\frac{r_\beta^k}{2k}<\infty,
\end{equation}
the deterministic series is absolutely convergent, and the centred independent
series is convergent in $L^2$ and almost surely.  We denote the limit by
$Y_\infty$; then $Y_\infty\in(0,\infty)$ almost surely.

We next show that the residual $X_m-X_{m,b}$ is negligible on the exponential
scale.  Since $X_{m,b}=\mathbb E(X_m\mid\mathcal F_{m,b})$,
\begin{equation}\label{eq:section5-conditional-var-identity}
  \mathbb E(X_m-X_{m,b})^2
  =\mathbb E X_m^2-\mathbb E X_{m,b}^2.
\end{equation}
The hypotheses of the lemma imply the second moment asymptotic \eqref{cm}; for
$\beta_m\to\beta$ the convergence is uniform by
Corollary~\ref{cor:uniform-second-ratio-L1}.  Hence
\begin{equation}\label{eq:section5-second-ratio}
  \frac{\mathbb E X_m^2}{(\mathbb E X_m)^2}
  \longrightarrow
  R_\beta:=\frac1{\sqrt{1-r_\beta}}.
\end{equation}
On the other hand, fix any $W_0>2$.  Lemma \ref{l35}, applied with $W=W_0$
and with $\beta_m$ in place of $\beta$, and then using the compact-uniform form in
Lemma~\ref{lem:cycle-uniformity} together with $\beta_m\to\beta$,
gives, for every fixed $b$ and after letting its truncation parameter $y\to\infty$,
\begin{equation}\label{eq:section5-truncated-second-lower}
  \liminf_{m\to\infty}\frac{\mathbb E X_{m,b}^2}{(\mathbb E X_m)^2}
  \ge (1-r_\beta^b)R_\beta .
\end{equation}
Combining \eqref{eq:section5-conditional-var-identity},
\eqref{eq:section5-second-ratio}, and \eqref{eq:section5-truncated-second-lower},
we obtain
\begin{equation}\label{eq:section5-residual-variance}
  \limsup_{m\to\infty}
  \frac{\mathbb E(X_m-X_{m,b})^2}{(\mathbb E X_m)^2}
  \le R_\beta r_\beta^b .
\end{equation}

We now prove the two tails in \eqref{lmb}.  The upper tail follows immediately
from Markov's inequality:
\begin{equation}\label{eq:section5-upper-tail}
  \mathbb P\left\{X_m>e^{m\varepsilon}\mathbb E X_m\right\}
  \le e^{-m\varepsilon}.
\end{equation}
For the lower tail, let $\alpha>0$ be arbitrary.  Choose $a>0$ such that
\[
  \mathbb P\{Y_\infty\le 4a\}<\alpha/3.
\]
Then choose $b$ so large that
\begin{equation}\label{eq:section5-choose-b}
  \mathbb P\{|Y_b-Y_\infty|>a\}<\alpha/3,
  \qquad
  \frac{4R_\beta r_\beta^b}{a^2}<\alpha/3.
\end{equation}
For this fixed $b$, the Portmanteau theorem, \eqref{eq:section5-Xmb-limit}, and \eqref{eq:section5-choose-b} give
\begin{align}\label{eq:section5-Xmb-lower-tight}
  \limsup_{m\to\infty}
  \mathbb P\left\{X_{m,b}<2a\,\mathbb E X_m\right\}
  &\le \mathbb P\{Y_b\le 2a\} \notag\\
  &\le \mathbb P\{Y_\infty\le 4a\}+\mathbb P\{|Y_b-Y_\infty|>a\}\le 2\alpha/3.
\end{align}
If $m$ is large enough that $e^{-m\varepsilon}<a$, then
\begin{align*}
\mathbb P\left\{X_m<e^{-m\varepsilon}\mathbb E X_m\right\}
&\le
\mathbb P\left\{X_{m,b}<2a\,\mathbb E X_m\right\} \\
&\quad+
\mathbb P\left\{|X_m-X_{m,b}|>a\,\mathbb E X_m\right\}.
\end{align*}
By Chebyshev's inequality and \eqref{eq:section5-residual-variance}, the limsup of
the second probability is at most $R_\beta r_\beta^b/a^2<\alpha/12$.  Together with
\eqref{eq:section5-Xmb-lower-tight}, this gives
\[
  \limsup_{m\to\infty}
  \mathbb P\left\{X_m<e^{-m\varepsilon}\mathbb E X_m\right\}
  \le \alpha .
\]
Since $\alpha>0$ is arbitrary, the lower tail tends to zero.  Combining this with
\eqref{eq:section5-upper-tail} proves \eqref{lmb}.  Finally,
\eqref{eq:section5-first-moment-log} and \eqref{lmb} imply \eqref{cfe}.
\end{proof}

\noindent\emph{Remark.}
The proof of Lemma \ref{l51} uses the alternatives in its statement only to
invoke the second moment ratio \eqref{cm} and the inequality \(r_\beta<1\).
Thus the same proof applies in any other parameter regime where Section
\ref{sm} establishes \eqref{cm} and \(r_\beta<1\).

\begin{lemma}\label{lem:beta-star-below-L1}
Let $L_1$ be defined by \eqref{eq:section5-L1}.  If $l\ge3$ and
\[
  \Phi'_{d,l}(L_1)\le0,
\]
then $\beta_*\le L_1$.
\end{lemma}

\begin{proof}
By Lemma \ref{lm21}, $\Phi'_{d,l}$ is strictly decreasing on $(0,1/d)$ and
$\beta_*$ is its unique zero.  Hence $\Phi'_{d,l}(L_1)\le0$ implies
$\beta_*\le L_1$.
\end{proof}

\noindent\textbf{Proof of Theorem \ref{thm}(2) in the $L_1$ subcase.}
By Theorem \ref{thm}(1), equivalently by the first moment asymptotic for the total
number of matchings proved in Section \ref{fm},
\begin{equation}\label{eq:section5-total-first-moment}
  \mathbb E Z=e^{m\Phi_{d,l}(\beta_*)+O(1)}.
\end{equation}
Let $\varepsilon>0$.  Markov's inequality gives
\[
  \mathbb P\left\{\frac1m\log Z>\Phi_{d,l}(\beta_*)+\varepsilon\right\}
  \le e^{-m(\Phi_{d,l}(\beta_*)+\varepsilon)}\mathbb E Z
  \longrightarrow0.
\]

For the lower bound, choose integers $h_m$ with $h_m/m\to\beta_*$ and, in the case
$l\ge3$, $h_m/m\le\beta_*$ for all $m$; for instance, take
$h_m=\lfloor m\beta_*\rfloor$.  Lemma \ref{lm21} gives
$\Phi_{d,l}(\beta_*)>0$.  If $l=2$, Lemma \ref{l51} applies directly.  In the $L_1$ subcase for
$l\ge3$ we assume $\beta_*\le L_1$, and therefore $h_m/m\le L_1$ for all
sufficiently large $m$.  Lemma \ref{l51} gives
\[
  \frac1m\log Z_{h_m}\longrightarrow\Phi_{d,l}(\beta_*)
  \qquad\text{in probability}.
\]
Since $Z\ge Z_{h_m}$,
\[
  \mathbb P\left\{\frac1m\log Z<\Phi_{d,l}(\beta_*)-\varepsilon\right\}
  \le
  \mathbb P\left\{\frac1m\log Z_{h_m}<\Phi_{d,l}(\beta_*)-\varepsilon\right\}
  \longrightarrow0.
\]
Combining the upper and lower tail estimates proves
\[
  \frac1m\log Z\longrightarrow\Phi_{d,l}(\beta_*)
\]
in probability in the $L_1$ subcase.  The remaining $L_2$ subcase is completed
in Section~\ref{ptc}, after Lemma~\ref{lem:fixed-density-L2}.\hfill\(\Box\)

\section{Weighted Free Energy}\label{wfe}

In this section we prove the weighted version of the free-energy convergence in the
$L_1$ regime.  Recall that
\[
  L_1:=\frac1{dl-d-l+2}=\frac1{(d-1)(l-1)+1}.
\]
For $x>0$ let
\[
  Z(x):=\sum_{h=0}^{H_m} Z_h x^h,\qquad H_m:=\lfloor m/d\rfloor,
\]
where $Z_h$ denotes the number of matchings with exactly $h$ hyperedges.  Define
\begin{equation}\label{eq:weighted-Delta}
  \Delta_{d,l}(\beta,x):=\Phi_{d,l}(\beta)+\beta\log x,
  \qquad 0<\beta<\frac1d .
\end{equation}
We use the continuous extension of $\Delta_{d,l}(\cdot,x)$ to $[0,1/d]$ when taking
suprema over densities.

\begin{lemma}\label{lem:weighted-maximizer}
For every $x>0$, the function $\Delta_{d,l}(\cdot,x)$ has a unique maximizer
$\beta_*(x)\in(0,1/d)$.  This maximizer is characterized by
\begin{equation}\label{eq:weighted-critical}
  \Phi_{d,l}'(\beta_*(x))+\log x=0,
\end{equation}
or equivalently
\begin{equation}\label{eq:weighted-critical-product}
  x(1-d\beta_*(x))^l=\beta_*(x)(1-\beta_*(x))^{l-1}.
\end{equation}
If $l\ge3$ and
\begin{equation}\label{eq:weighted-L1-condition}
  \Phi_{d,l}'(L_1)+\log x\le0,
\end{equation}
then
\begin{equation}\label{cdbs}
  \beta_*(x)\le L_1 .
\end{equation}
\end{lemma}

\begin{proof}
Differentiating \eqref{eq:weighted-Delta},
\[
  \partial_\beta \Delta_{d,l}(\beta,x)
  =-\log\beta-(l-1)\log(1-\beta)+l\log(1-d\beta)+\log x,
\]
and
\[
  \partial_{\beta\beta}^2\Delta_{d,l}(\beta,x)
  =-\frac1\beta-
  \frac{ld-d\beta-(l-1)}{(1-\beta)(1-d\beta)}<0,
  \qquad 0<\beta<\frac1d .
\]
Moreover,
\[
  \lim_{\beta\downarrow0}\partial_\beta\Delta_{d,l}(\beta,x)=+\infty,
  \qquad
  \lim_{\beta\uparrow1/d}\partial_\beta\Delta_{d,l}(\beta,x)=-\infty .
\]
Hence $\partial_\beta\Delta_{d,l}(\beta,x)$ has a unique zero in $(0,1/d)$, and
strict concavity makes this zero the unique maximizer.  The equation for this zero is
\eqref{eq:weighted-critical}, which is equivalent to
\eqref{eq:weighted-critical-product}.  Since the continuous extension satisfies
$\Delta_{d,l}(0,x)=0$ and the right derivative at $0$ is $+\infty$, the maximum value is
strictly positive.  Finally, if $l\ge3$ and \eqref{eq:weighted-L1-condition} holds, then
the strictly decreasing function $\partial_\beta\Delta_{d,l}(\beta,x)$ is already
non-positive at $L_1$.  Since its unique zero is $\beta_*(x)$, we obtain \eqref{cdbs}.
\end{proof}

\begin{lemma}\label{lem:weighted-first-moment}
For every fixed $x>0$,
\begin{equation}\label{eq:weighted-first-moment}
  \frac1m\log \mathbb E Z(x)
  \longrightarrow
  \Delta_{d,l}(\beta_*(x),x).
\end{equation}
\end{lemma}

\begin{proof}
Let $\overline\Phi_{d,l}$ denote the continuous extension of $\Phi_{d,l}$ to
$[0,1/d]$.  We shall use the following endpoint-uniform estimate.  It follows
from \eqref{ezb} and the Stirling bounds in Lemma~\ref{l12}, exactly as in the
proof of Theorem~\ref{thm}(1):
\begin{equation}\label{eq:weighted-crude-first-moment}
  \sup_{0\le h\le H_m}
  \left|
  \frac1m\log \mathbb E Z_h-
  \overline\Phi_{d,l}\left(\frac hm\right)
  \right|=o(1),
\end{equation}
where $Z_0=1$ and hence the term $h=0$ is interpreted with
$\overline\Phi_{d,l}(0)=0$.  Indeed, one may use the uniform form
$n!=\exp\{n\log n-n+O(\log(n+1))\}$, with $0!=1$, in each factorial in
\eqref{ezb}; the total error is $O(\log m)$, including the endpoint regimes
$h=O(1)$ and $H_m-h=O(1)$.  Therefore
\begin{align*}
\mathbb E Z(x)
&=\sum_{h=0}^{H_m}x^h\mathbb E Z_h  \\
&\le (H_m+1)
\exp\left\{m\sup_{0\le\beta\le1/d}
       \Delta_{d,l}(\beta,x)+o(m)\right\}  \\
&=(H_m+1)\exp\left\{m\Delta_{d,l}(\beta_*(x),x)+o(m)\right\} .
\end{align*}
This gives the upper bound for the logarithmic limit.

For the lower bound, choose any integer sequence $h_m$ with
$h_m/m\to\beta_*(x)$, for instance $h_m=\lfloor m\beta_*(x)\rfloor$.
Since $\beta_*(x)\in(0,1/d)$, the compact-uniform first-moment estimate proved
in Section~\ref{fm} gives
\[
  \frac1m\log\mathbb E Z_{h_m}\longrightarrow \Phi_{d,l}(\beta_*(x)).
\]
Consequently,
\[
  \frac1m\log\left(x^{h_m}\mathbb E Z_{h_m}\right)
  \longrightarrow
  \Phi_{d,l}(\beta_*(x))+\beta_*(x)\log x
  =\Delta_{d,l}(\beta_*(x),x).
\]
Since $\mathbb E Z(x)\ge x^{h_m}\mathbb E Z_{h_m}$, the lower bound follows.
\end{proof}

\noindent\textbf{Proof of Theorem \ref{thm}(3) in the $L_1$ subcase.}
The proof below covers the case $l=2$ and the $L_1$ case $l\ge3$ with
\eqref{eq:weighted-L1-condition}.  The $L_2$ subcase, and hence the full
certified-threshold formulation with $x\le x_{\rm cert}$, is completed in
Section~\ref{ptc}.  Write
\[
  \beta_x:=\beta_*(x),\qquad
  \Delta_x:=\Delta_{d,l}(\beta_x,x)
  =\Phi_{d,l}(\beta_x)+\beta_x\log x .
\]
Let $\varepsilon>0$.  By Markov's inequality and Lemma~\ref{lem:weighted-first-moment},
\begin{align}\label{eq:weighted-upper-tail}
\mathbb P\left\{\frac1m\log Z(x)>\Delta_x+\varepsilon\right\}
&\le e^{-m(\Delta_x+\varepsilon)}\mathbb E Z(x)
\longrightarrow 0 .
\end{align}

It remains to prove the lower tail.  Choose
\[
  h_m:=\lfloor m\beta_x\rfloor,\qquad \beta_m:=\frac{h_m}{m}.
\]
Then $\beta_m\to\beta_x$.  If $l\ge3$, condition
\eqref{eq:weighted-L1-condition} and Lemma~\ref{lem:weighted-maximizer} give
$\beta_x\le L_1$, and hence $\beta_m\le L_1$ for all sufficiently large $m$.  If
$l=2$, no upper restriction on $\beta_m$ is needed.  Lemma~\ref{l51} therefore yields
\begin{equation}\label{eq:weighted-fixed-density-lower}
  \frac1m\log Z_{h_m}
  \longrightarrow
  \Phi_{d,l}(\beta_x)
  \qquad\text{in probability}.
\end{equation}
Since $Z(x)\ge x^{h_m}Z_{h_m}$,
\[
\mathbb P\left\{\frac1m\log Z(x)<\Delta_x-\varepsilon\right\}
\le
\mathbb P\left\{\frac1m\log Z_{h_m}+\beta_m\log x<\Delta_x-\varepsilon\right\}.
\]
Because $\beta_m\log x\to\beta_x\log x$, the event on the right is contained,
for all large $m$, in
\[
\left\{\frac1m\log Z_{h_m}<\Phi_{d,l}(\beta_x)-\frac{\varepsilon}{2}\right\}.
\]
This probability tends to $0$ by \eqref{eq:weighted-fixed-density-lower}.
Combining this lower-tail estimate with
\eqref{eq:weighted-upper-tail} proves
\[
  \frac1m\log Z(x)
  \longrightarrow
  \Phi_{d,l}(\beta_*(x))+\beta_*(x)\log x
\]
in probability.\hfill\(\Box\)

\section{Another Criterion to Guarantee Global Maxima}\label{ptc}

This section proves the additional $L_2$ density regime used in the certified
threshold $L_{\rm cert}$ in Theorem~\ref{thm}.  The argument below replaces the
geometric proof in the original draft by a one-dimensional reduction.  Throughout the section we assume \(d\ge2\), \(l\ge3\), and
\(0<\beta<1/d\).  We use the coordinates
\[
  \eta:=\theta/l,\qquad
  \widehat\Psi_\beta(\rho,\eta):=\Psi_{d,l}(\beta,\rho,l\eta),
\]
and
\[
  \widehat{\mathcal R}_\beta
  :=\{(\rho,\eta):(\rho,l\eta)\in\mathcal R_\beta\}
  =\left\{0<\rho<\beta,\,
  \max\left\{0,2\beta-\rho-\frac1d\right\}<\eta<\beta-\rho\right\}.
\]
Set
\begin{align}
A_{d,l}&:=\frac1d\left(1-
        \sqrt{\frac{d-1}{d^{l/(l-1)}-1}}\right),\label{eq:A-dl}\\[2mm]
B_{d,l}&:=\frac{dl+l^2-2l-d+1}{2dl^2-dl},\label{bca1}\\[2mm]
C_{d,l}&:=\frac1{1+\sqrt{(d-1)(l-1)}}.
\end{align}
Thus
\begin{equation}\label{eq:L2-section7}
  L_2=\min\{A_{d,l},B_{d,l},C_{d,l}\}.
\end{equation}
The following lemma records the elementary algebraic consequences of the definition
of $L_2$ that will be used below.  We keep the three constants
$A_{d,l},B_{d,l},C_{d,l}$ in the statement because $L_2$ is one of the two certified
interval endpoints entering $L_{\rm cert}$, although the one-dimensional reduction
below uses only the consequences listed here.

\begin{lemma}[Consequences of the $L_2$ bound]\label{lem:L2-elementary-consequences}
Assume $d\ge2$, $l\ge3$, and $0<\beta\le L_2$.  Then:
\begin{enumerate}
\item $\beta<C_{d,l}$, and the Hessian at
      $(\beta^2,l(d-1)\beta^2)$ is negative definite.
\item $\beta\le A_{d,l}\le 1/(2l)$, and the function
\[
  A(\rho):=\frac{a'(\rho)}{a(\rho)}
  =\frac1l\left\{\frac1\rho-
  \frac{2(l-1)}{\beta-\rho}-\frac{l-1}{1-2\beta+\rho}\right\}
\]
from \eqref{sol} is positive on $(0,\beta^2]$.
\item
\begin{equation}\label{eq:cpl-derivative}
  \left.\partial_\rho\widehat\Psi_\beta(\rho,\eta)\right|_{(\rho,\eta)=(d\beta^2,0)}
  =(l-1)\log\frac{1-2\beta+d\beta^2}{(1-d\beta)^2}-\log d\le0.
\end{equation}
\item
\begin{equation}\label{eq:Phi-positive-L2}
  \Phi_{d,l}(\beta)>0,
  \qquad 0<\beta\le L_2.
\end{equation}
\end{enumerate}
\end{lemma}

\begin{proof}
A direct comparison of the first and third entries in \eqref{eq:L2-section7}
gives $A_{d,l}<C_{d,l}$ for $d\ge2,l\ge3$.  Thus $L_2<C_{d,l}$.  Hence
$\beta<C_{d,l}$, and the negative definiteness of the Hessian follows from
Lemma~\ref{l24}.

Next, $A_{d,l}\le1/(2l)$.  If $d\ge2l$ this is automatic from the definition of
$A_{d,l}$.  If $d<2l$, it is equivalent, after squaring two positive quantities, to
\[
  4l^2(d-1)\ge (2l-d)^2\bigl(d^{l/(l-1)}-1\bigr).
\]
This last inequality follows from the mean-value theorem applied to
$u\mapsto u^{l/(l-1)}$ on $[1,d]$ and the elementary bound
$\frac{l}{l-1}d^{1/(l-1)}(2l-d)^2\le4l^2$.  Thus $\beta\le1/(2l)$.  On
$(0,\beta^2]$,
\[
  A'(\rho)=-\frac1l\left\{\frac1{\rho^2}+
  \frac{2(l-1)}{(\beta-\rho)^2}-\frac{l-1}{(1-2\beta+\rho)^2}\right\}<0,
\]
because $\beta\le1/(2l)$ implies
$\beta-\rho\le\beta<\sqrt2(1-2\beta+\rho)$.  Therefore $A$ is decreasing on this
interval, and
\[
  \beta^2(1-\beta)^2 l A(\beta^2)=1-2l\beta+l\beta^2>0.
\]
Hence $A(\rho)>0$ for $0<\rho\le\beta^2$.

For \eqref{eq:cpl-derivative}, direct substitution into
$\partial_\rho\widehat\Psi_\beta$ gives the displayed expression.  Exponentiating the
inequality in \eqref{eq:cpl-derivative} shows that it is equivalent to
\[
  1-2\beta+d\beta^2\le d^{1/(l-1)}(1-d\beta)^2 .
\]
Solving this quadratic inequality in $\beta$ gives precisely
$0<\beta\le A_{d,l}$.

Finally, another direct calculation gives $\Phi'_{d,l}(A_{d,l})>0$.  Since
$\Phi'_{d,l}$ is strictly decreasing, $0<\beta\le L_2\le A_{d,l}$ implies
$\Phi'_{d,l}(u)>0$ for every $u\in(0,\beta]$.  With
$\Phi_{d,l}(0+)=0$, this yields \eqref{eq:Phi-positive-L2}.
\end{proof}

\begin{lemma}\label{lem:diagonal-derivative}
Assume \(0<\beta\le A_{d,l}\).  Along the line segment
\[
  \eta=d\beta^2-\rho,
  \qquad \beta^2\le \rho\le d\beta^2,
\]
one has
\begin{equation}\label{eq:g-diagonal-nonpositive}
  \partial_\rho\widehat\Psi_\beta(\rho,d\beta^2-\rho)\le0.
\end{equation}
The inequality is strict except at \(\rho=\beta^2\) and possibly at the endpoint
\(\rho=d\beta^2\).
\end{lemma}

\begin{proof}
Put
\[
  g(\rho):=\partial_\rho\widehat\Psi_\beta(\rho,d\beta^2-\rho),
  \qquad \beta^2\le\rho\le d\beta^2.
\]
At the saddle, \(g(\beta^2)=0\).  By \eqref{eq:cpl-derivative},
\(g(d\beta^2)\le0\).  Moreover
\begin{align*}
  g'(\rho)
  &=\partial_{\rho\rho}^2\widehat\Psi_\beta(\rho,d\beta^2-\rho)
    -\partial_{\rho\eta}^2\widehat\Psi_\beta(\rho,d\beta^2-\rho)\\
  &=-\frac1\rho+\frac{l-1}{1-2\beta+\rho}
    +\frac{2(l-1)}{\beta-\rho}
   =\frac{V(\rho)}{\rho(\beta-\rho)(1-2\beta+\rho)},
\end{align*}
where
\[
  V(\rho)=\rho^2+(l-2\beta-l\beta)\rho+2\beta^2-\beta.
\]
The denominator is positive on \([\beta^2,d\beta^2]\).  The polynomial \(V\) is an
upward-opening quadratic and \(V(0)=-\beta(1-2\beta)<0\); hence \(V\) has exactly one
positive zero.  Thus \(g\) decreases and then increases on the interval.  A function
with this monotonicity pattern has its maximum on a compact interval at one of the two
endpoints.  Since both endpoint values are at most zero, \eqref{eq:g-diagonal-nonpositive}
follows.
\end{proof}

\begin{lemma}\label{lem:vertical-reduction-L2}
Assume \(0<\beta\le L_2\).  For each \(\rho\in(0,\beta)\), the function
\(\eta\mapsto\widehat\Psi_\beta(\rho,\eta)\) has a unique maximizer
\(u_\beta(\rho)\) in the vertical section of \(\widehat{\mathcal R}_\beta\).  If
\[
  M_\beta(\rho):=\widehat\Psi_\beta(\rho,u_\beta(\rho)),
\]
then \(M_\beta\) is strictly increasing on \((0,\beta^2)\) and strictly decreasing on
\((\beta^2,\beta)\).  Consequently,
\[
  M_\beta(\rho)\le M_\beta(\beta^2)=2\Phi_{d,l}(\beta),
\]
with equality only at \(\rho=\beta^2\).
\end{lemma}

\begin{proof}
For fixed \(\rho\), the derivative in the \(\eta\) direction satisfies
\[
  \partial_{\eta\eta}^2\widehat\Psi_\beta(\rho,\eta)<0.
\]
Moreover, \(\partial_\eta\widehat\Psi_\beta\to+\infty\) at the lower endpoint of
the vertical section and
\(\partial_\eta\widehat\Psi_\beta\to-\infty\) at the upper endpoint
\(\eta=\beta-\rho\).  Hence the vertical maximizer \(u_\beta(\rho)\) exists and is
unique.  It is characterized by
\begin{equation}\label{eq:vertical-critical}
  (d-1)(\beta-\rho-u_\beta(\rho))^2
  =u_\beta(\rho)\{1-2\beta d+\rho d+d u_\beta(\rho)\}.
\end{equation}
The implicit function theorem applies, and the envelope theorem gives
\begin{equation}\label{eq:Mprime-envelope}
  M_\beta'(\rho)=\partial_\rho\widehat\Psi_\beta(\rho,u_\beta(\rho)).
\end{equation}

We first treat \(0<\rho<\beta^2\).  Let \(a(\rho)\) be as in \eqref{sol}, and set
\[
  R(\rho):=\beta-\rho-(d-1)a(\rho),
  \qquad
  S(\rho):=1-2\beta d+\rho d+d(d-1)a(\rho).
\]
Since \(A(\rho)=a'(\rho)/a(\rho)>0\) on \((0,\beta^2]\), the function \(R\) is decreasing
there, and so
\[
  R(\rho)>R(\beta^2)=\beta(1-d\beta)>0.
\]
If \(S(\rho)\le0\), then
\[
  (d-1)R(\rho)^2-(d-1)a(\rho)S(\rho)>0.
\]
If \(S(\rho)>0\), then \(S(\rho')>0\) for every
\(\rho'\in[\rho,\beta^2]\), because \(S'(\rho')=d+d(d-1)a'(\rho')>0\) on
\((0,\beta^2]\).  Hence the logarithm
\[
  h(\rho):=\log a(\rho)+\log S(\rho)-2\log R(\rho)
\]
is well-defined on \([\rho,\beta^2]\).  On this interval,
\[
  h'(\rho')
  =\frac d{S(\rho')}+\frac2{R(\rho')}
  +A(\rho')\left(1+\frac{d(d-1)a(\rho')}{S(\rho')}
  +\frac{2(d-1)a(\rho')}{R(\rho')}\right)>0.
\]
Since \(h(\beta^2)=0\), we have \(h(\rho)<0\) for \(\rho<\beta^2\) in this case.  Thus in
all cases
\begin{equation}\label{eq:F-at-a-positive}
  (d-1)R(\rho)^2-(d-1)a(\rho)S(\rho)>0,
  \qquad 0<\rho<\beta^2.
\end{equation}
For fixed \(\rho\), the left-hand side of the vertical equation,
\[
  F_\rho(\eta):=(d-1)(\beta-\rho-\eta)^2
  -\eta(1-2\beta d+\rho d+d\eta),
\]
is strictly decreasing in \(\eta\) on the feasible vertical section.  Since
\eqref{eq:F-at-a-positive} says that
\(F_\rho((d-1)a(\rho))>0\), the zero \(u_\beta(\rho)\) of \(F_\rho\) satisfies
\begin{equation}\label{eq:u-greater-a}
  u_\beta(\rho)>(d-1)a(\rho),
  \qquad 0<\rho<\beta^2.
\end{equation}
Using \eqref{eq:vertical-critical} to simplify
\(\partial_\rho\widehat\Psi_\beta(\rho,u_\beta(\rho))\), we get
\begin{equation}\label{eq:Mprime-a-formula}
  M_\beta'(\rho)=l\log\frac{u_\beta(\rho)}{(d-1)a(\rho)}.
\end{equation}
Together with \eqref{eq:u-greater-a}, this proves \(M_\beta'(\rho)>0\) on
\((0,\beta^2)\).

We next treat \(\rho>\beta^2\).  Let
\[
  z_0:=(\beta^2,(d-1)\beta^2).
\]
If \(u_\beta(\rho)\ge(d-1)\beta^2\), then
\[
  \partial_\rho\widehat\Psi_\beta(\beta^2,u_\beta(\rho))<0,
\]
because \(\partial_\rho\widehat\Psi_\beta(z_0)=0\) and
\(\partial_{\rho\eta}^2\widehat\Psi_\beta<0\).  Since
\(\partial_{\rho\rho}^2\widehat\Psi_\beta<0\) by Lemma~\ref{lm33}, it follows that
\(M_\beta'(\rho)<0\).

It remains to consider the case \(u_\beta(\rho)<(d-1)\beta^2\).  We first note that
\(u_\beta(\rho)\) lies above the line \(\eta=d\beta^2-\rho\).  If
\(d\beta^2-\rho\le0\) this is immediate.  Otherwise, evaluating \(F_\rho\) at
\(d\beta^2-\rho\) gives
\[
  F_\rho(d\beta^2-\rho)=(1-d\beta)^2(\rho-\beta^2)>0.
\]
Since \(F_\rho\) is strictly decreasing and \(F_\rho(u_\beta(\rho))=0\), we obtain
\(u_\beta(\rho)>d\beta^2-\rho\).  Therefore
\[
  \rho_0:=d\beta^2-u_\beta(\rho)
\]
belongs to \([\beta^2,d\beta^2]\), and
\((\rho_0,u_\beta(\rho))\) lies on the diagonal segment considered in
Lemma~\ref{lem:diagonal-derivative}.  That lemma gives
\[
  \partial_\rho\widehat\Psi_\beta(\rho_0,u_\beta(\rho))
  =\partial_\rho\widehat\Psi_\beta(\rho_0,d\beta^2-\rho_0)\le0.
\]
Using again \(\partial_{\rho\rho}^2\widehat\Psi_\beta<0\) and
\(\rho\ge\rho_0\), we obtain
\[
  M_\beta'(\rho)=\partial_\rho\widehat\Psi_\beta(\rho,u_\beta(\rho))<0.
\]
Thus \(M_\beta\) is strictly decreasing on \((\beta^2,\beta)\).

At \(\rho=\beta^2\), the vertical critical equation has the solution
\(u_\beta(\beta^2)=(d-1)\beta^2\), and hence
\(M_\beta(\beta^2)=\widehat\Psi_\beta(\beta^2,(d-1)\beta^2)=2\Phi_{d,l}(\beta)\).
The asserted uniqueness follows from the strict monotonicity on the two sides of
\(\beta^2\).
\end{proof}

\begin{lemma}\label{lem:L2-global-maximum}
Assume \(d\ge2\), \(l\ge3\), and \(0<\beta\le L_2\).  Then
\begin{equation}\label{eq:L2-global-max}
  \sup_{(\rho,\theta)\in\mathcal R_\beta}
  \Psi_{d,l}(\beta,\rho,\theta)=2\Phi_{d,l}(\beta),
\end{equation}
and the unique maximizer is
\[
  (\rho,\theta)=(\beta^2,l(d-1)\beta^2).
\]
Moreover, the Hessian at this point is negative definite and
\begin{equation}\label{eq:L2-r-less-one}
  r_\beta:=(d-1)(l-1)\left(\frac{\beta}{1-\beta}\right)^2<1.
\end{equation}
\end{lemma}

\begin{proof}
The identity \eqref{eq:L2-global-max} and uniqueness follow immediately from
Lemma~\ref{lem:vertical-reduction-L2}, since every point in
\(\widehat{\mathcal R}_\beta\) lies in one of the vertical sections.  The negative
definiteness of the Hessian and \eqref{eq:L2-r-less-one} follow from
\(\beta\le L_2<C_{d,l}\) and Lemma~\ref{l24}.
\end{proof}

\begin{lemma}\label{lem:second-moment-nondegenerate-global}
Let \(0<\beta<1/d\) and assume that
\begin{enumerate}
\item \(\Phi_{d,l}(\beta)>0\);
\item the point
      \[
        (\beta^2,l(d-1)\beta^2)
      \]
      is the unique global maximizer of the continuous extension of
      \(\Psi_{d,l}(\beta,\cdot,\cdot)\) to
      \(\overline{\mathcal R_\beta}\); and
\item the Hessian \(H(\beta,\beta^2,l(d-1)\beta^2)\) is negative definite.
\end{enumerate}
Then, along subsequences for which \(m\beta\in\mathbb N\),
\begin{equation}\label{eq:second-ratio-general}
  \lim_{m\to\infty}
  \frac{\mathbb E Z_{m\beta}^2}{(\mathbb E Z_{m\beta})^2}
  =\frac{1-\beta}{\sqrt{\beta^2(d+l-dl)-2\beta+1}}.
\end{equation}
\end{lemma}

\begin{proof}
This is exactly the discrete two-dimensional Laplace estimate used in the proof of
Lemma~\ref{l08}.  The proof of that lemma uses the hypotheses there only to guarantee
three facts: the saddle value is \(2\Phi_{d,l}(\beta)\), the saddle is the unique global
maximizer, and the Hessian at the saddle is negative definite.  These are precisely the
assumptions above.  The boundary and off-saddle terms are exponentially negligible because the
maximum on \(\overline{\mathcal R_\beta}\) is unique and because
\(\Phi_{d,l}(\beta)>0\) rules out endpoint contributions with the same exponential
order.  The local Gaussian computation at the saddle is unchanged.
It gives the constant in \eqref{eq:second-ratio-general}.
\end{proof}

\begin{lemma}[Local stability of a non-degenerate global maximum]
\label{lem:local-stability-global-max}
Let \(0<\beta_0<1/d\), and put
\[
  z_\beta:=(\beta^2,l(d-1)\beta^2),\qquad
  r_\beta:=(d-1)(l-1)\left(\frac{\beta}{1-\beta}\right)^2 .
\]
Assume that \(\Phi_{d,l}(\beta_0)>0\), that \(r_{\beta_0}<1\), that
\(z_{\beta_0}\) is the unique global maximizer of the continuous extension of
\(\Psi_{d,l}(\beta_0,\cdot,\cdot)\) to \(\overline{\mathcal R_{\beta_0}}\), and that
\(H(\beta_0,z_{\beta_0})\) is negative definite.  Then there is an open interval
\(I\subset(0,1/d)\) containing \(\beta_0\) such that, for every \(\beta\in I\),
\(z_\beta\) is the unique global maximizer of the continuous extension of
\(\Psi_{d,l}(\beta,\cdot,\cdot)\) to \(\overline{\mathcal R_\beta}\),
\(H(\beta,z_\beta)\) is negative definite, \(\Phi_{d,l}(\beta)>0\), and
\(r_\beta<1\).  Moreover, the second-moment ratio in
\eqref{eq:second-ratio-general} holds uniformly for \(\beta\in I\), along
subsequences for which \(m\beta\in\mathbb N\).
\end{lemma}

\begin{proof}
The functions \(\Phi_{d,l}\), \(r_\beta\), \(z_\beta\), and the entries of
\(H(\beta,z_\beta)\) are continuous in \(\beta\).  Thus, after restricting to a
small open interval \(I_0\subset(0,1/d)\) around \(\beta_0\), the inequalities
\(\Phi_{d,l}(\beta)>0\), \(r_\beta<1\), and the negative definiteness of
\(H(\beta,z_\beta)\) hold throughout \(I_0\).  Also, \(z_\beta\) remains an
interior critical point of \(\Psi_{d,l}(\beta,\cdot,\cdot)\) by the critical
point equations \eqref{eq3}--\eqref{eq2}.

We next prove the uniform exponent gap away from the saddle.  Fix
\(\delta>0\) small enough that \(z_\beta\in\mathcal R_\beta\) and
\(B_{2\delta}(z_\beta)\subset\mathcal R_\beta\) for all \(\beta\) in a possibly
smaller neighbourhood of \(\beta_0\).  Suppose that no such gap holds outside
\(B_\delta(z_\beta)\).  Then there exist \(\beta_n\to\beta_0\) and
\(w_n\in\overline{\mathcal R_{\beta_n}}\) such that
\[
  \operatorname{dist}(w_n,z_{\beta_n})\ge\delta,
  \qquad
  \Psi_{d,l}(\beta_n,w_n)
  \ge 2\Phi_{d,l}(\beta_n)-o(1).
\]
The sets \(\overline{\mathcal R_{\beta_n}}\) are contained in the fixed compact
rectangle \([0,1/d]\times[0,l/d]\).  Passing to a subsequence, \(w_n\to w_0\).
The defining inequalities of the closed feasible regions are continuous in
\(\beta\), so \(w_0\in\overline{\mathcal R_{\beta_0}}\).  The convention
\(x\log x=0\) gives a jointly continuous extension of \(\Psi_{d,l}\) to these
compact closures, and hence
\[
  \Psi_{d,l}(\beta_0,w_0)=2\Phi_{d,l}(\beta_0).
\]
The assumed uniqueness of the maximizer at \(\beta_0\) gives
\(w_0=z_{\beta_0}\).  This contradicts the facts that \(w_n\to w_0\),
\(z_{\beta_n}\to z_{\beta_0}\), and
\(\operatorname{dist}(w_n,z_{\beta_n})\ge\delta\).  Therefore, after shrinking
\(I_0\) to an interval \(I\), for every such \(\delta\) there exists
\(\eta_\delta>0\) such that
\[
  \sup\{\Psi_{d,l}(\beta,w):
        w\in\overline{\mathcal R_\beta},
        \operatorname{dist}(w,z_\beta)\ge\delta\}
  \le 2\Phi_{d,l}(\beta)-\eta_\delta,
  \qquad \beta\in I .
\]

It remains only to identify the maximizer inside \(B_\delta(z_\beta)\).  By
continuity of the Hessian, we may choose \(\delta\) and then shrink \(I\) so
that the Hessian of \(\Psi_{d,l}(\beta,\cdot,\cdot)\) is negative definite on
\(B_\delta(z_\beta)\), uniformly for \(\beta\in I\).  Hence
\(\Psi_{d,l}(\beta,\cdot,\cdot)\) is strictly concave on this ball.  Since
\(z_\beta\) is an interior critical point, it is the unique maximizer in the
ball.  Together with the uniform exponent gap outside the ball, this proves that
\(z_\beta\) is the unique global maximizer on \(\overline{\mathcal R_\beta}\) for
all \(\beta\in I\).

The uniform version of \eqref{eq:second-ratio-general} follows from the same
compact-uniform discrete Laplace estimate used in Lemma~\ref{l08}: the exponent
gap above controls the complement of a fixed neighbourhood of the saddle, while
the local Hessian and all Stirling prefactors vary continuously and remain
bounded away from zero and infinity on compact subsets of the neighbourhood.
\end{proof}

\begin{corollary}[Fixed-density convergence under a stable saddle]
\label{cor:stable-saddle-fixed-density}
Assume the hypotheses of Lemma~\ref{lem:local-stability-global-max} at a point
\(\beta_0\in(0,1/d)\).  If \(h_m\in\mathbb N\) and \(h_m/m\to\beta_0\), then
\[
  \frac1m\log Z_{h_m}\longrightarrow \Phi_{d,l}(\beta_0)
\]
in probability.
\end{corollary}

\begin{proof}
For all large \(m\), \(\beta_m:=h_m/m\) lies in the interval \(I\) supplied by
Lemma~\ref{lem:local-stability-global-max}.  That lemma gives the second-moment
ratio uniformly for \(\beta_m\), and also gives \(r_{\beta_m}<1\).  The
compact-uniform cycle estimates of Lemma~\ref{lem:cycle-uniformity} apply to the
same sequence \(\beta_m\).  Repeating the subgraph-conditioning proof of
Lemma~\ref{l51}, with these uniform inputs and with \(h_m\) in place of
\(m\beta\), gives
\[
  \frac1m\log Z_{h_m}-\frac1m\log\mathbb E Z_{h_m}\longrightarrow0
\]
in probability.  Lemma~\ref{lem:first-moment-local} gives
\((1/m)\log\mathbb E Z_{h_m}\to\Phi_{d,l}(\beta_0)\).  This proves the claim.
\end{proof}

\begin{lemma}\label{lem:fixed-density-L2}
Let \(h_m\in\mathbb N\), put \(\beta_m=h_m/m\), and assume
\(\beta_m\to\beta\in(0,1/d)\).  Suppose \(d\ge2\), \(l\ge3\), and
\(\beta_m\le L_2\) for all sufficiently large \(m\).  Then
\begin{equation}\label{eq:fixed-density-L2}
  \frac1m\log Z_{h_m}\longrightarrow \Phi_{d,l}(\beta)
\end{equation}
in probability.
\end{lemma}

\begin{proof}
By \eqref{eq:Phi-positive-L2} and Lemma~\ref{lem:L2-global-maximum}, the hypotheses of
Lemma~\ref{lem:second-moment-nondegenerate-global} hold uniformly on compact subsets of
\(\{0<\beta\le L_2\}\).  Hence the second moment ratio
\eqref{eq:second-ratio-general} holds with \(\beta_m\) in place of \(\beta\).  In addition,
\eqref{eq:L2-r-less-one} gives \(r_\beta<1\).  The proof of Lemma~\ref{l51} applies
verbatim with this second-moment input: the alternatives in Lemma~\ref{l51} are used
only to invoke the second moment ratio and the inequality \(r_\beta<1\).  Therefore
\eqref{eq:fixed-density-L2} follows.
\end{proof}

\noindent\textbf{Proof of Theorem \ref{thm}(3) in the $L_2$ subcase.}
Let \(x>0\) and assume
\begin{equation}\label{eq:weighted-L2-condition}
  \Phi_{d,l}'(L_2)+\log x\le0.
\end{equation}
Let \(\beta_*(x)\) be the unique maximizer of
\[
  \Delta_{d,l}(\beta,x)=\Phi_{d,l}(\beta)+\beta\log x,
  \qquad 0<\beta<1/d.
\]
Since \(\Phi_{d,l}'\) is strictly decreasing, \eqref{eq:weighted-L2-condition} implies
\begin{equation}\label{eq:beta-x-below-L2}
  \beta_*(x)\le L_2.
\end{equation}
Put
\[
  \Delta_*(x):=\Delta_{d,l}(\beta_*(x),x).
\]
By Lemma~\ref{lem:weighted-first-moment},
\begin{equation}\label{eq:section7-weighted-first-moment}
  \frac1m\log\mathbb E Z(x)\longrightarrow \Delta_*(x).
\end{equation}
Thus Markov's inequality gives, for every \(\varepsilon>0\),
\[
  \mathbb P\left\{\frac1m\log Z(x)>\Delta_*(x)+\varepsilon\right\}
  \le e^{-m(\Delta_*(x)+\varepsilon)}\mathbb E Z(x)\longrightarrow0.
\]
For the lower tail, let
\[
  h_m:=\lfloor m\beta_*(x)\rfloor.
\]
Then \(h_m/m\to\beta_*(x)\) and, by \eqref{eq:beta-x-below-L2}, \(h_m/m\le L_2\) for all
large \(m\).  Lemma~\ref{lem:fixed-density-L2} gives
\[
  \frac1m\log Z_{h_m}\longrightarrow \Phi_{d,l}(\beta_*(x))
  \qquad\text{in probability}.
\]
Since \(Z(x)\ge x^{h_m}Z_{h_m}\),
\[
  \liminf_{m\to\infty}\frac1m\log Z(x)
  \ge \Phi_{d,l}(\beta_*(x))+\beta_*(x)\log x
  =\Delta_*(x)
\]
in probability.  Combining the upper and lower tails proves
\[
  \frac1m\log Z(x)\longrightarrow
  \Phi_{d,l}(\beta_*(x))+\beta_*(x)\log x
\]
in probability in the $L_2$ subcase.

Taking $x=1$ in the same argument proves Theorem~\ref{thm}(2) whenever
$l\ge3$ and $\beta_*\le L_2$.  Combining this with the $L_1$ subcase proved in
Section~\ref{fe} proves Theorem~\ref{thm}(2) under the single condition
$\beta_*\le L_{\rm cert}$.  Similarly, the $L_1$ proof in Section~\ref{wfe}
and the $L_2$ proof above prove Theorem~\ref{thm}(3) whenever
$\beta_*(x)\le L_{\rm cert}$, equivalently whenever $0<x\le x_{\rm cert}$.
\hfill\(\Box\)

\bigskip
\noindent\textbf{Proof of Theorem \ref{th12}.}
Assume that
\[
  G_{\beta_*}(s)>0,
  \qquad 0<s<\beta_*-c_0(\beta_*),
\]
where \(\beta_*\) is the root of \eqref{eq1}.  Lemma~\ref{le33}, together with
the boundary argument in the proof of Lemma~\ref{le34}, implies that
\((\beta_*^2,l(d-1)\beta_*^2)\) is the unique global maximizer of the continuous
extension of \(\Psi_{d,l}(\beta_*,\cdot,\cdot)\) to
\(\overline{\mathcal R_{\beta_*}}\).  The Hessian at the saddle is negative definite:
if \(l\le d\) this follows from Lemma~\ref{l24}(1), while
if \(l>d\), writing \(q_*=\beta_*/(1-\beta_*)\), the root equation becomes
\[
  (1-(d-1)q_*)^l=q_*.
\]
Let \(q_0=((d-1)(l-1))^{-1/2}\) and set
\(u=(d-1)q_0=\sqrt{(d-1)/(l-1)}\).  Since \(l>d\), we have \(0<u<1\),
and \(q_0=1/(u(l-1))\).  Hence
\[
  (1-(d-1)q_0)^l=(1-u)^l
  \le (1+u)^{-l}<\frac1{lu}<\frac1{(l-1)u}=q_0,
\]
where we used \((1-u)(1+u)\le1\) and \((1+u)^l>lu\).  Thus
\((1-(d-1)q_0)^l<q_0\).  Since the function
\(q\mapsto (1-(d-1)q)^l-q\) is strictly decreasing on
\((0,(d-1)^{-1})\), we have \(q_*<q_0\).  Lemma~\ref{l24}(2) applies.  Also
\(\Phi_{d,l}(\beta_*)>0\), since \(\beta_*\) is the unique maximizer of
\(\Phi_{d,l}\) and \(\Phi_{d,l}(0+)=0\).
Moreover \(r_{\beta_*}<1\) by the displayed inequality, or trivially from
\(\beta_*<1/d\le C_{d,l}\) when \(l\le d\).  Lemma~\ref{lem:local-stability-global-max}
therefore gives an interval \(I\) containing \(\beta_*\) on which the hypotheses of
Lemma~\ref{lem:second-moment-nondegenerate-global} hold uniformly, and on which
\(r_\beta<1\).  Consequently, the hypotheses of
Corollary~\ref{cor:stable-saddle-fixed-density} hold at \(\beta_*\).  Therefore,
for every integer sequence \(h_m\) with \(h_m/m\to\beta_*\),
\[
  \frac1m\log Z_{h_m}\longrightarrow \Phi_{d,l}(\beta_*)
\]
in probability.  Finally, the proof of Theorem~\ref{thm}(2) in Section~\ref{fe}
applies verbatim:
the upper tail is controlled by the total first moment, and the lower tail follows
from \(Z\ge Z_{\lfloor m\beta_*\rfloor}\).  Hence
\[
  \frac1m\log Z\longrightarrow \Phi_{d,l}(\beta_*)
\]
in probability.
\hfill\(\Box\)

\section{Maximum Matching Size}\label{mm}

In this section we record a simple consequence of the first-moment estimate for large
matchings.  The result concerns the maximum cardinality of a matching; this is different
from an inclusion-maximal matching.  Let
\[
  H_m:=\left\lfloor \frac md\right\rfloor,
  \qquad
  Z_{\ge k}:=\sum_{h=\lceil k\rceil}^{H_m} Z_h,
  \qquad
  \nu_m:=\max\{h: Z_h>0\}.
\]
Assume throughout this section that
\[
 f_l\left(\frac1d\right)<0 .
\]
By Lemma~\ref{lm21}, there is a unique
\(\beta_0\in(\beta_*,1/d)\) such that
\(\Phi_{d,l}(\beta_0)=0\), and
\(\Phi'_{d,l}(\beta_0)<0\).  Put
\[
  a_0:=-\Phi'_{d,l}(\beta_0)>0
\]
and define the real centering
\begin{equation}
  K_m:=m\beta_0+\frac{\log m}{2\Phi'_{d,l}(\beta_0)}
      =m\beta_0-\frac{\log m}{2a_0} .
\label{dkm}
\end{equation}
Since \(K_m\) need not be an integer, the correct statement is the following integer-sequence
version.

\begin{proposition}\label{prop:local-large-matching-first-moment}
Let \(k_m\in\{0,1,\ldots,H_m\}\) be any integer sequence such that
\[
  r_m:=k_m-K_m=O(1).
\]
Then
\[
  \mathbb E Z_{k_m}
  =\frac{\exp\{\Phi'_{d,l}(\beta_0)r_m+o(1)\}}
  {\sqrt{2\pi\beta_0(1-d\beta_0)}} .
\]
In particular, if \(r_m\to r\), then
\[
  \mathbb E Z_{k_m}
  \longrightarrow
  \frac{e^{\Phi'_{d,l}(\beta_0)r}}
  {\sqrt{2\pi\beta_0(1-d\beta_0)}} .
\]
\end{proposition}

\begin{proof}
Set \(\beta_m=k_m/m\).  Then
\[
  \beta_m=\beta_0+\frac{\log m}{2m\Phi'_{d,l}(\beta_0)}+\frac{r_m}{m},
\]
so \(\beta_m\to\beta_0\).  Lemma~\ref{lem:first-moment-local}, equivalently
\eqref{azh}, gives
\[
  \mathbb E Z_{k_m}
  =\frac{e^{m\Phi_{d,l}(\beta_m)}}
  {\sqrt{2\pi m\beta_m(1-d\beta_m)}}(1+o(1)).
\]
Taylor expansion at \(\beta_0\), using \(\Phi_{d,l}(\beta_0)=0\), gives
\[
  m\Phi_{d,l}(\beta_m)
  =\Phi'_{d,l}(\beta_0)
   \left(\frac{\log m}{2\Phi'_{d,l}(\beta_0)}+r_m\right)
   +O\left(\frac{(\log m)^2}{m}\right)
  =\frac{\log m}{2}+\Phi'_{d,l}(\beta_0)r_m+o(1).
\]
The factor \(e^{(\log m)/2}\) cancels the prefactor \(m^{-1/2}\), and the result follows.
\end{proof}

\begin{proposition}\label{prop:large-matching-tail}
Let \(k_m=\lfloor K_m\rfloor\).  Then
\[
  \lim_{C\to\infty}\limsup_{m\to\infty}\mathbb E Z_{\ge k_m+C}=0.
\]
Consequently,
\[
  \lim_{C\to\infty}\limsup_{m\to\infty}
  \mathbb P\{\nu_m\ge k_m+C\}=0.
\]
\end{proposition}

\begin{proof}
Choose \(\varepsilon_0>0\) so small that
\([\beta_0/2,\beta_0+\varepsilon_0]\subset(0,1/d)\).  By Lemma~\ref{lem:first-moment-local},
there is a constant \(A<\infty\) such that, uniformly for
\(h/m\in[\beta_0/2,\beta_0+\varepsilon_0]\),
\begin{equation}
  \mathbb E Z_h
  \le A m^{-1/2}\exp\left\{m\Phi_{d,l}(h/m)\right\} .
\label{eq:tail-local-first-moment}
\end{equation}
Since \(\Phi_{d,l}\) is concave and \(\Phi_{d,l}(\beta_0)=0\), for every
\(u\ge\beta_0\),
\begin{equation}
  \Phi_{d,l}(u)\le \Phi'_{d,l}(\beta_0)(u-\beta_0)=-a_0(u-\beta_0).
\label{eq:tail-concavity}
\end{equation}
Let
\[
  M_m:=\left\lceil \frac{5\log m}{2a_0}\right\rceil .
\]
For \(C\le \delta<M_m\), put \(h=k_m+\delta\).  Then
\(h-K_m=\delta+O(1)\), and the same Taylor expansion used in
Proposition~\ref{prop:local-large-matching-first-moment} gives, uniformly in this range,
\[
  \mathbb E Z_{k_m+\delta}
  \le A_1 e^{-a_0\delta}
\]
for a constant \(A_1<\infty\).  Hence
\begin{equation}
  \sum_{\delta=C}^{M_m-1}\mathbb E Z_{k_m+\delta}
  \le A_2 e^{-a_0 C}.
\label{eq:near-tail-bound}
\end{equation}
For \(h\ge k_m+M_m\) and \(h/m\le\beta_0+\varepsilon_0\), we have
\(h-m\beta_0\ge 2a_0^{-1}\log m+O(1)\).  Combining
\eqref{eq:tail-local-first-moment} and \eqref{eq:tail-concavity},
\[
  \mathbb E Z_h\le A_3 m^{-5/2}
\]
for all sufficiently large \(m\), uniformly over such \(h\).  Thus the contribution of these
indices is \(O(m^{-3/2})\).  Finally, on the remaining interval
\(h/m>\beta_0+\varepsilon_0\), strict concavity and the fact that \(\Phi_{d,l}<0\) on
\((\beta_0,1/d)\) imply that \(\Phi_{d,l}\le-\eta\) for some \(\eta>0\); the crude
Stirling bounds used in Section~\ref{fm} then show that the total contribution from this
range is exponentially small.  Together with \eqref{eq:near-tail-bound}, this proves the
first assertion.  The probability bound follows from Markov's inequality,
\[\mathbb P\{\nu_m\ge k_m+C\}\le \mathbb E Z_{\ge k_m+C}.\]
\end{proof}

\appendix

\section{\texorpdfstring{A Finiteness Observation for the Condition in Theorem~\ref{th12}}{A Finiteness Observation for the Condition in Theorem 1.2}}\label{sect:A}

In this appendix we justify the statement that the hypothesis in Theorem~\ref{th12} can
hold for only finitely many parameter pairs.  Recall that the corrected definition used in
Section~\ref{sm} is
\[
  t_\beta(s)=\frac{(\beta-s)^{1/l}s^{1-2/l}}{(1-\beta-s)^{(l-1)/l}},
  \qquad 0<s<\beta .
\]
The condition in Theorem~\ref{th12} is
\begin{equation}
  G_{\beta_*}(s)>0,\qquad 0<s<\beta_*-c_0(\beta_*),
\label{eq:appendix-G-condition}
\end{equation}
where \(c_0\) is given in Section~\ref{sm}.

\begin{proposition}\label{prop:finite-G-condition}
The condition \eqref{eq:appendix-G-condition} can hold for only finitely many pairs
\((d,l)\) with \(d,l\ge2\).
\end{proposition}

\begin{proof}
Suppose that \((d_n,l_n)\) is a sequence with \(d_n,l_n\ge2\) and
\(\max\{d_n,l_n\}\to\infty\).  We shall show that, after passing to a subsequence, the
condition \eqref{eq:appendix-G-condition} fails for all sufficiently large \(n\).  This rules
out infinitely many admissible pairs.

Write \(\beta_n=\beta_*(d_n,l_n)\), and set
\[
  z_n:=\frac{1-d_n\beta_n}{1-\beta_n}.
\]
The equation defining \(\beta_*\) is equivalent to
\begin{equation}
  (d_n-1)z_n^{l_n}+z_n-1=0,
  \qquad
  \beta_n=\frac{1-z_n}{d_n-z_n}.
\label{eq:z-beta-relation}
\end{equation}
We repeatedly use the explicit expression for \(G_\beta\) from Section~\ref{sm}.  All
error terms below are along the chosen subsequence.

First assume that
\[
  \frac{\log d_n}{l_n}\longrightarrow \alpha\in(0,\infty).
\]
Then \eqref{eq:z-beta-relation} gives
\[
  z_n=e^{-\alpha+o(1)},
  \qquad
  d_n\beta_n=1-e^{-\alpha}+o(1),
  \qquad
  \frac{c_0(\beta_n)}{\beta_n}=o(1).
\]
Fix any \(\zeta\in(0,1)\), and put \(s_n=\zeta\beta_n\).  Then
\(0<s_n<\beta_n-c_0(\beta_n)\) for all large \(n\), and
\[
  t_{\beta_n}(s_n)=\frac{\zeta e^\alpha(1-e^{-\alpha})}{d_n}(1+o(1)).
\]
Substitution in the formula for \(G_\beta\) gives the following expansion.  The
terms not containing \(t_{\beta_n}(s_n)\) contribute
\[
  \frac{2e^{-\alpha}(1-e^{-\alpha})}{d_n}+l_n s_n(1-e^{-\alpha})+O(l_ns_n^2),
\]
whereas the terms containing \(t_{\beta_n}(s_n)\) contribute
\[
  -l_n t_{\beta_n}(s_n)
  \left\{d_ns_n^2+s_ne^{-\alpha}+\frac{2e^{-\alpha}(1-e^{-\alpha})}{d_n}
\right\}
  +o(l_n/d_n).
\]
Using \(s_n=\zeta(1-e^{-\alpha})d_n^{-1}(1+o(1))\) and the displayed expression for
\(t_{\beta_n}(s_n)\), these terms combine to
\begin{equation}
  G_{\beta_n}(s_n)
  =-\frac{l_n}{d_n}\,
  \zeta^2(1-\zeta)(1-e^{-\alpha})^2(e^\alpha-1)(1+o(1))<0
\label{eq:G-alpha-finite}
\end{equation}
for all sufficiently large \(n\).  Thus \eqref{eq:appendix-G-condition} fails in this case.

Next assume that
\[
  \frac{\log d_n}{l_n}\longrightarrow0.
\]
Then \(l_n\to\infty\), \(z_n\to1\), \(d_n\beta_n\to0\),
\(c_0(\beta_n)/\beta_n=o(1)\), and
\begin{equation}
  u_n:=l_n(d_n-1)\beta_n\longrightarrow\infty .
\label{eq:u-to-infty}
\end{equation}
Indeed, if \(w_n=1-z_n\), then \eqref{eq:z-beta-relation} gives
\(w_n=(d_n-1)(1-w_n)^{l_n}\); this implies \(w_n\to0\) and
\(l_nw_n\to\infty\), and \eqref{eq:u-to-infty} follows from
\(\beta_n=w_n/(d_n-1+w_n)\).  Fix \(\zeta\in(0,1)\) and put
\(s_n=\zeta\beta_n\).  Then \(s_n\) is admissible for all large \(n\),
\[
  t_{\beta_n}(s_n)=\zeta\beta_n(1+o(1)),
\]
and the leading terms in \(G_{\beta_n}(s_n)\) are those of order
\(l_n(d_n-1)\beta_n^2\).  More explicitly, the first line of \(G_\beta\) is
\(O(d_nl_ns_n^2)=O(l_nd_n\beta_n^2)\), the second line contributes
\(2\beta_n\{1-(d_n-1)(l_n-1)t_{\beta_n}(s_n)\}+o(l_nd_n\beta_n^2)\), and the
third line contributes
\(s_n\{2l_n(d_n-1)t_{\beta_n}(s_n)+l_nd_n\beta_n\}+o(l_nd_n\beta_n^2)\).
After inserting \(s_n=\zeta\beta_n\) and
\(t_{\beta_n}(s_n)=\zeta\beta_n(1+o(1))\), the terms of smaller order cancel or are
absorbed into the error, and
\begin{equation}
  G_{\beta_n}(s_n)
  =l_n(d_n-1)\beta_n^2\,\zeta(\zeta-1)(1+o(1))<0.
\label{eq:G-small-ratio}
\end{equation}
This again contradicts \eqref{eq:appendix-G-condition}.

It remains to consider the case
\[
  \frac{\log d_n}{l_n}\longrightarrow\infty.
\]
Put \(a_n=d_n^{-1/l_n}\).  Then \(a_n\to0\), and
\eqref{eq:z-beta-relation} gives
\begin{equation}
  z_n=a_n(1+o(1)),
  \qquad
  d_n\beta_n=1-a_n+o(a_n).
\label{eq:superexp-beta}
\end{equation}
If \(l_n\ge3\) eventually, take \(\zeta=1/3\) and set \(s_n=\zeta\beta_n\).  The formula
for \(c_0\) gives \(c_0(\beta_n)/\beta_n=(2l_n-1)^{-1}+o(1)\), so \(s_n\) is admissible.
Moreover
\[
  t_{\beta_n}(s_n)
  =(1-\zeta)^{1/l_n}\zeta^{1-2/l_n}\beta_n^{1-1/l_n}(1+o(1)).
\]
Substitution into \(G_\beta\), using \eqref{eq:superexp-beta}, yields the leading
term
\[
  s_n t_{\beta_n}(s_n)(d_n-1)
  \{-\beta_n l_n+3\beta_nd_n+2l_n-3\beta_nd_nl_n+o(l_n)\},
\]
while all terms not multiplied by \(t_{\beta_n}(s_n)\) are smaller by a factor
\(a_n^{1/l_n}+o(1)\).  Since \(d_n\beta_n=1-a_n+o(a_n)\), this is equivalent to
\begin{equation}
  G_{\beta_n}(s_n)
  =\zeta\,t_{\beta_n}(s_n)
    \left((l_n-1)\zeta-(l_n-2)+o(l_n)\right)<0,
\label{eq:G-superexp-lge3}
\end{equation}
for all sufficiently large \(n\), because \(\zeta=1/3\) makes the coefficient negative for
all \(l_n\ge3\).

Finally, if the subsequence has \(l_n=2\), choose
\(\zeta_n=a_n^{1/2}=d_n^{-1/4}\) and set \(s_n=\zeta_n\beta_n\).  Then
\(s_n/\beta_n\to0\), so \(s_n<\beta_n-c_0(\beta_n)\) for all large \(n\).  A direct
substitution in \(G_\beta\), using \eqref{eq:superexp-beta} with \(l_n=2\), gives
\begin{equation}
  G_{\beta_n}(s_n)=-a_n^2(1+o(1))<0.
\label{eq:G-superexp-l2}
\end{equation}
Thus the condition \eqref{eq:appendix-G-condition} fails in every possible divergent
sequence of parameter pairs.  Therefore it can hold for only finitely many pairs
\((d,l)\).
\end{proof}

\bigskip
\bigskip

\noindent{\textbf{Acknowledgements.}} Z.L.'s research is supported by NSF grant DMS-1608896 and Simons grant 683143.

\bibliography{matching_cleaned_references}
\bibliographystyle{plain}
\end{document}